\documentclass[reqno,final,10pt]{amsart}
\usepackage{natbib}  %Nature-like bibliography
\usepackage{fancyhdr} %Headers
\usepackage{color} %Color definition
\usepackage{hyperref} %Internal and external links
\usepackage{graphicx} %Graphics inclusion

\usepackage{geometry} 
\usepackage{tikz}

\usepackage{amsmath,amsthm,amssymb,latexsym,bbm,mathtools} 
\usepackage{url} % for formatting references

\usepackage[notcite,notref]{showkeys} % shows labels
\usepackage[latin1]{inputenc}
\usepackage[british,english]{babel} 
\usepackage{paralist}

\definecolor{aleacolor}{rgb}{0.16,0.59,0.78}

\hypersetup{
  breaklinks,
  colorlinks=true,
  linkcolor=aleacolor,
  urlcolor=aleacolor,
  citecolor=aleacolor}

\renewcommand{\cite}{\citet}

\theoremstyle{plain}
\newtheorem{theorem}{Theorem}[section]                                          
\newtheorem{proposition}[theorem]{Proposition}                          
\newtheorem{lemma}[theorem]{Lemma}

\theoremstyle{definition}
\newtheorem{definition}[theorem]{Definition}
\theoremstyle{remark}
\newtheorem{remark}[theorem]{Remark}
\newtheorem{example}[theorem]{Example}

\makeatletter \@addtoreset{equation}{section} \makeatother

\renewcommand{\Pr}{\mathbb{P}}

\newcommand{\Z}{\mathbb{Z}}
\newcommand{\R}{\mathbb{R}}
\newcommand{\E}{\mathbb{E}}

\newcommand{\XX}{\mathcal{X}} 
\newcommand{\WW}{\mathcal{W}} 
\newcommand{\YY}{\mathcal{Y}} 
\newcommand{\RR}{\mathcal{R}}

\newcommand{\abs}[1]{\lvert#1\rvert} %Definition Betragsstriche
\newcommand{\Abs}[1]{\left\lvert#1\right\rvert}
\newcommand{\tn}{\textnormal}
\newcommand{\wt}{\widetilde}
\newcommand{\wh}{\widehat}

\begin{document}

\title[Diffusion ratchets with locally negative drift]{Asymptotic properties of
  certain %protein translocation \\ 
  diffusion ratchets \\ with  locally negative drift} 

\author{Andrej Depperschmidt} 
\author{Sophia G\"otz} 

\address{University of Freiburg,  %\newline   
  Abteilung f\"ur Mathematische Stochastik \newline  
  Eckerstra\ss e 1, %\newline  
  79104 Freiburg, Germany }
\email{depperschmidt@stochastik.uni-freiburg.de,   sophiagoetz@gmx.net}  
% \urladdr{\url{http://www.stochastik.uni-freiburg.de/homepages/deppers/}} 

% \address{University of Freiburg \newline   
% Abteilung f\"ur Mathematische Stochastik \newline  
% Eckerstra\ss e 1\newline  
% 79104 Freiburg, Germany }

% \email{sophiagoetz@gmx.net}  

\thanks{Research supported by the BMBF,
    Germany, through FRISYS (Freiburg Initiative for Systems biology),
    Kennzeichen 0313921.}


\begin{abstract}
  We consider two reflecting diffusion processes $(X_t)_{t \ge 0}$ with a moving
  reflection boundary given by a non-decreasing pure jump Markov process
  $(R_t)_{t \ge 0}$. Between the jumps of the reflection boundary the diffusion
  part behaves as a reflecting Brownian motion with negative drift or as a
  reflecting Ornstein-Uhlenbeck process. In both cases at rate $\gamma(X_t
  -R_t)$ for some $\gamma \ge 0$ the reflection boundary jumps to a new
  value chosen uniformly in $[R_{t-},X_t]$. Since after each jump of the
  reflection boundary the diffusions are reflected at a higher level we call the
  processes \emph{Brownian ratchet} and \emph{Ornstein-Uhlenbeck ratchet}. Such
  diffusion ratchets are biologically motivated by passive protein transport
  across membranes. The processes considered here are generalisations of the
  Brownian ratchet (without drift) studied in
  \citep{DepperschmidtPfaffelhuber:2010}. For both processes we prove a law of
  large numbers, in particular each of the ratchets moves to infinity at a
  positive speed which can be computed explicitly, and a central limit
  theorem.
\end{abstract}

\maketitle

\section{Introduction}
\label{sec:introduction}

Reflecting diffusion processes constitute an important class of stochastic
processes that appear in various applications. We consider two particular
examples of a \emph{diffusion ratchet} variants of which were introduced in
\citep{SimonPeskinOster1992} and \citep{PeskinOdellOster:1993} motivated by
protein transport across cell membranes. Generally speaking a diffusion ratchet
is a diffusion process reflected at a non-decreasing jump process. The name
\emph{ratchet} is justified by the fact that each jump prevents the diffusion
from attaining lower values. In a sense a jump of the reflection boundary
process can be thought of as a \emph{click} of a ratchet. In \citep{MR2223904} a
diffusion ratchet (modelling a molecular motor) in which a particle moves
according to a Brownian motion between equally spaced (deterministic) barriers
is studied. The particle can cross such barriers from left to right but is
reflected if it hits a barrier to its left. The two models that we consider in
the present note are both generalisations of the diffusion ratchet studied in
\citep{DepperschmidtPfaffelhuber:2010}. In the model studied there a particle
moves according to a reflecting Brownian motion and the reflection boundary
jumps a rate proportional to the distance between the particle and the current
reflection boundary. At jump times the new reflection boundary is chosen
uniformly between the old one and the position of the particle.

\medskip 
Let us now introduce the models that we consider here, then briefly explain the
biological motivation and finally state our results.

\subsection{The models} 
\label{sec:model}

Let $(\XX, \RR) \coloneqq (X_t, R_t)_{t\geq 0}$ be a time-homogeneous Markov
process starting in $(X_0, R_0) = (x_0,0)$, for some $x_0\ge 0$. Here $\XX =
(X_t)_{t\ge 0}$ is a diffusion process reflected at a non-decreasing jump
process $\RR = (R_t)_{t \ge 0}$. Given that $(\XX, \RR)$ is in $(X_t,R_t)$ at
time $t$, the reflection boundary process jumps at rate $\gamma(X_t-R_t)$ for
some $\gamma \ge 0$. If $\tau$ is a jump time of the reflection boundary then
the new position is uniformly distributed on the interval $[R_{\tau-}, X_\tau]$.
By this dynamics, $R_t\leq X_t$ for all $t\geq 0$, almost surely. In principle
the above description works with any reflecting diffusion between the jumps. As
mentioned earlier we study here two particular cases and to distinguish them we
write $(\XX, \RR)$ and $(\wh \XX, \wh \RR)$ for the corresponding processes.

\medskip    
\begin{enumerate}
\item[(I)] If for $\mu \ge 0$ the process $\XX = (X_t)_{t\ge 0}$ is a Brownian  
  motion with negative infinitesimal drift $-\mu$, unit variance (see
  Section~\ref{sec:rbmd}) and reflection boundary process $\RR= (R_t)_{t \ge 0}$
  then we refer to the process $(\mathcal X, \mathcal R)$ as \emph{the
    $(\gamma,\mu)$-Brownian ratchet}. 
 
\item[(II)] If for $\mu \ge 0$ the process $\wh\XX = (\wh X_t)_{t\ge 0}$ is an
  Ornstein-Uhlenbeck process with infinitesimal drift $-\mu x$, unit variance
  (see Section~\ref{sec:refl-OU}) and reflection process $\wh \RR=(\wh R_t)_{t
    \ge 0}$ then we refer to $(\wh \XX, \wh \RR)$ as \emph{the
    $(\gamma,\mu)$-Ornstein-Uhlenbeck ratchet}.
\end{enumerate}

 Whenever we want to stress the dependence on the parameters, we write
  $(\mathcal X^{(\gamma,\mu)}, 
  \mathcal R^{(\gamma,\mu)}) = (X_t^{(\gamma,\mu)}, R_t^{(\gamma,\mu)})_{t\geq
    0}$ for the $(\gamma,\mu)$-Brownian ratchet and $(\wh{\mathcal 
    X}^{(\gamma,\mu)}, \wh{\mathcal R}^{(\gamma,\mu)}) = (\wh
  X_t^{(\gamma,\mu)}, \wh R_t^{(\gamma,\mu)})_{t\geq 0}$ for the
  $(\gamma,\mu)$-Ornstein-Uhlenbeck ratchet. 

\begin{figure}[h]
%  \centering
  \begin{tikzpicture}
    \draw[thick,rounded corners=5pt] 
    (0,1.48) -- (0,0) -- (3,0) -- (3,3) -- (0,3) -- (0,1.62);  
    
    \draw[line width=1pt] 
    (-1.7,2.4) .. controls (0,1) and (-0.6,0.5) .. 
    (-0.8,0.4)..  controls (-1,0.3) and (-1.6,0.2) ..  
    (-1.3,0.6) .. controls (-1,1) and (0.2,1.7) .. 
    (-0.2,2) ..   controls (-0.6,2.3) and (-1.1,2.6) .. 
    (-1,1.9)  ..  controls (-0.9,1.5) and (-0.2,1.65) .. 
    (0,1.55) ..   controls (0.2,1.45) and (0.6,0.7) .. 
    (0.7,0.8) ..  controls (0.8,0.9) and (0.8,2) .. 
    (1,1.9) ..    controls (1.2,1.8) and (1.35,1.6) ..
    (1.5,1.4) ..  controls (1.65,1.2) and (1.7,2) .. (1.8,2.5);   
    
    \filldraw [gray] (0.2,0.3) circle (1.5pt)
                     (0.8,0.4) circle (1.5pt)
                     (1.1,2.5) circle (1.5pt) 
                     (1.6,1)   circle (1.5pt)
                     (0.3,2)   circle (1.5pt)
                     (0.41,1.07)   circle (1.5pt)
                     (1.07,1.85)   circle (1.5pt)
                     (2.6,0.9) circle (1.5pt)
                     (2.3,2.7) circle (1.5pt); 
    
                     % \draw (0.8,2.8) node  {{\small \textbf{\textit{inside}}}}
                     % (-1.8,2.8) node  {{\small \textbf{\textit{outside}}}}; 
                     
    \draw[thick,rounded corners=5pt] 
    (7,1.48) -- (7,0) -- (10,0) -- (10,3) -- (7,3) -- (7,1.62);  
    
    \draw[line width=1pt] (5,1.55) -- (9.5,1.55);

    \draw[line width=0.5pt] (9.5,1.80)--(9.5,1.30);
    \draw[line width=0.5pt] (8,1.80)--(8,1.55);
    
    \draw[<->] (8.05,1.75) -- (9.45,1.75); 
    \draw[<->] (7.05,1.35) -- (9.45,1.35); 
    
    \draw[<-] (5.05,1.75) --(6.9,1.75); 
    \draw (8.8,1.95) node  {$R_t$} 
    (8.2,1.1)   node {$X_t$}
    (5.7,1.95) node {\textit{drift}} ;
    
    \filldraw [gray] (9.2,0.3)    circle (1.5pt)
                     (7.8,0.4)    circle (1.5pt)
                     (8.1,2.5)    circle (1.5pt) 
                     (8.6,1)      circle (1.5pt)
                     (8,1.55)     circle (1.5pt)
                     (9,1.55)     circle (1.5pt)
                     (9.3,2)      circle (1.5pt)
                     (9.41,1.07)  circle (1.5pt)
                     (7.5,2.75)   circle (1.5pt); 
\end{tikzpicture}
\caption{A diagram of the ratcheting mechanism for protein transport on the left
and of a reflecting diffusion with negative local drift on the right. %The negative
%drift arises because the protein must be unfolded during the transport.
}
\label{fig:protein}
\end{figure}
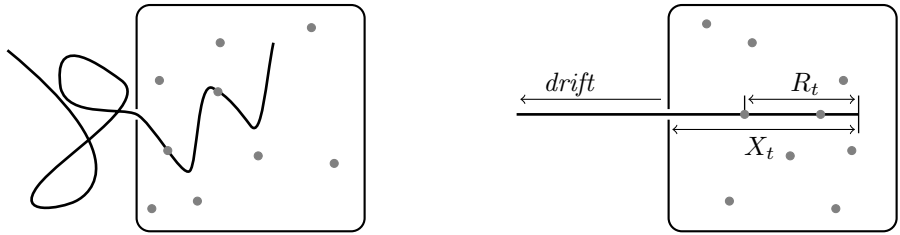

\subsection{Biological motivation}
\label{sec:biol-motiv}
Inside a typical cell different proteins are involved in many processes. They
usually need to be transported after or during the \emph{translation} 
(production) to various locations at which they are required. Depending on the
protein and its functions there are different transport mechanisms. In the
present paper we focus on the \emph{passive} protein transport across membranes
of e.g.\ endoplasmic reticulum (ER) or mitochondria for which ratcheting models
were introduced by \citet{SimonPeskinOster1992} and
\citet{PeskinOdellOster:1993}. The main idea in these models is that due to
thermal fluctuations the protein moves, say inside and outside the ER for
definiteness, through a nanopore in the membrane according to a diffusion; see
the left part of Figure~\ref{fig:protein}. Inside the ER, ratcheting molecules
can bind to the protein at a certain rate. These ratcheting molecules are too
big (in our model they are actually infinitesimally small but one can imagine
that binding of the molecules leads to a deformation of the protein at
the ratcheting sites) to pass through the nanopore and prevent the protein from
diffusing outside the ER, i.e.\ the protein performs a reflected diffusion with
jumping reflection boundary which is due to binding of new ratcheting molecules.
In the last two decades such models have been studied extensively in biology,
physics as well as in mathematics. For a detailed overview of the recent
literature and for more biological motivation we refer to
\citep{DepperschmidtKettererPfaffelhuber:2012} and references therein.

With this motivation in mind $X_t$ (and $\wh X_t$) can be interpreted as the
length of the protein inside ER at time $t$ and $R_t$ (and $\wh R_t$) as the
distance between the ``head'' of the protein and the ratcheting molecule closest
to the nanopore; see the right part of Figure~\ref{fig:protein}. Since typically
proteins have to be unfolded during translocation into ER, the movement inside
takes place against a force pointing outside which explains the locally negative
drift of the ratchets. 

\subsection{Results}
\label{sec:results}

For both, the $(\gamma,\mu)$-Brownian ratchet and the
$(\gamma,\mu)$-Ornstein-Uhlenbeck ratchet we prove a law of large numbers as
well as a central limit theorem. Furthermore we compute the speed of the
ratchets in terms of the Airy $Ai$-function in the case of Brownian ratchet and
in terms of the Tricomi confluent hypergeometric function in the case of
Ornstein-Uhlenbeck ratchet.

\begin{theorem}[LLN and CLT for the Brownian ratchet]  
  \label{theoremRBM}
  Let $(\XX,\RR) = (X_t,R_t)_{t\geq 0}$ be the $(\gamma,\mu)$-Brownian ratchet
  % with negative local drift
  starting in $(x_0,0)$ with $x_0 \geq 0$. If
  $\gamma,\mu \ge 0$ then  
  \begin{align}
    \label{eq:cgamma}
    \frac{X_t}{t} \xrightarrow{t\to\infty} v(\mu,\gamma) \coloneqq  -\frac{\gamma^{1/3}}{2^{2/3}}
    \frac{Ai'((2\gamma)^{-2/3}\mu^2)}{Ai((2\gamma)^{-2/3}\mu^2)} - \frac12 \mu
    \quad \text{almost surely,} 
  \end{align}
  where $Ai(\cdot)$ is the Airy function.  
  Furthermore in the case $\gamma>0$ there is $\sigma=\sigma(\mu,\gamma)>0$ such
  that   
  \[\frac{X_t-t v(\mu,\gamma)}{\sigma \sqrt{t}} \xRightarrow{t\to\infty} 
  X.\] Here `` $\xRightarrow{}$'' denotes convergence in distribution and $X$ is a
  standard Gaussian random variable.
\end{theorem}

Note that though the result is formulated for $\mu \ge 0$ for the proof we only
need to consider the case $\mu>0$. In the case $\mu=0$ the
$(\gamma,\mu)$-Brownian ratchet as well as the$(\gamma,\mu)$-Ornstein-Uhlenbeck
ratchet reduce to the process studied in \citep{DepperschmidtPfaffelhuber:2010}.

\begin{theorem}[LLN and CLT for the Ornstein-Uhlenbeck ratchet]  
  \label{theoremOU}
  Assume $\mu >0$ and $\gamma \ge 0$. Let $(\wh\XX,\wh\RR) = (\wh X_t, \wh
  R_t)_{t\geq 0}$ be the $(\gamma,\mu)$-Ornstein-Uhlenbeck ratchet starting in
  $(x_0,0)$ for $x_0\geq  0$. For $x \ge 0$ set
  \begin{align}
    \label{eq:57}
    h_{\mu,\gamma}(x) \coloneqq e^{-\gamma x/\mu -\mu x^2} U
  \left(\frac12-\frac{\gamma^2}{4\mu^3},\frac12,
    \Bigl(\frac\gamma{\mu^{3/2}}+\sqrt\mu x\Bigr)^2 \right),    
  \end{align}
  where $U$ is the Tricomi confluent hypergeometric function (see \eqref{eq:U}
  for a definition).  
  Then 
  \begin{align}\label{eq:cgamma-ou}
    \frac{\wh X_t}{t} \xrightarrow{t\to\infty} \wh v(\mu,\gamma) \coloneqq  -
    \frac{h'_{\mu,\gamma}(0)}{2h_{\mu,\gamma}(0)} - \frac{\mu \int_0^\infty
      h_{\mu,\gamma}(x)\, dx }{h_{\mu,\gamma}(0)}     \quad \text{almost surely.}    
  \end{align}
  Furthermore in the case $\gamma>0$ there is
  $\wh\sigma=\wh\sigma(\mu,\gamma)>0$ such that  
  \begin{align*}
    \frac{\wh X_t - t \wh v(\mu,\gamma)}{\wh\sigma \sqrt{t}}  
  \xRightarrow{t\to\infty} X 
  \end{align*} 
  for a standard Gaussian random variable $X$.  
\end{theorem}

\begin{figure}[h]
  \begin{center}
    \includegraphics[width=0.75\textwidth]{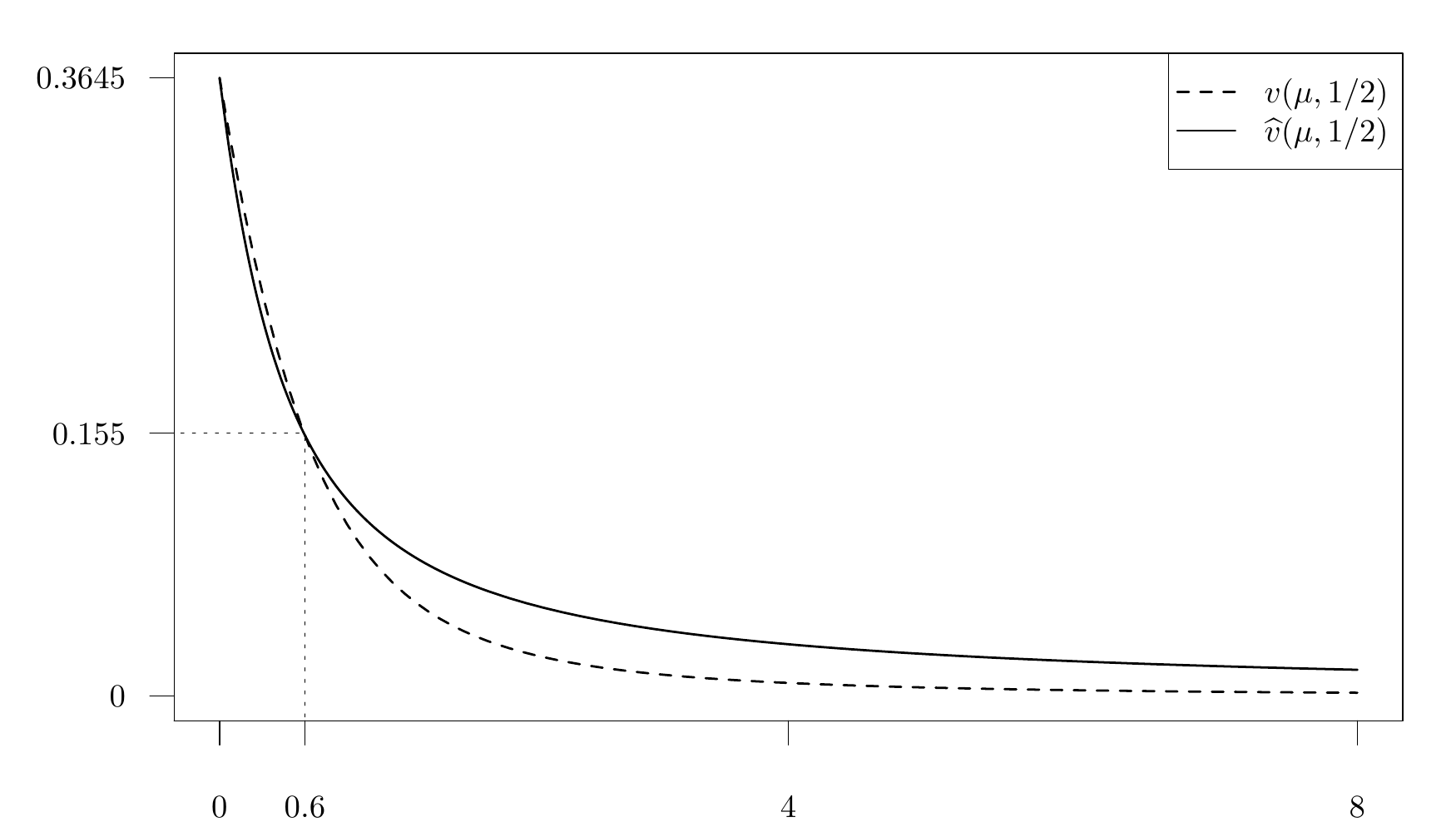}%  
    \caption{\label{fig:speed-compar} Comparison of the speed of the ratchets for
      $\gamma=\tfrac12$ and $\mu \in [0,8]$. } 
  \end{center}
\end{figure}
\begin{remark}[Comparison of the ratchets]
  In Figure~\ref{fig:speed-compar} we plot the speed of both ratchets in the
  interval $\mu \in [0,8]$ in the case $\gamma=1/2$. The plots are based on
  numerical computations using \textsc{Mathematica}. In the neighbourhood of
  zero, here approximately in the interval $(0,0.6)$, the Brownian ratchet is
  faster whereas outside that interval the Ornstein-Uhlenbeck ratchet has a
  higher speed. Heuristically this can be explained: If $X_t - R_t \approx \wh 
  X_t- \wh R_t$ are large and $\mu$ small then the drift of $X_t$ towards $R_t$
  is smaller than that of $\wh X_t$ towards $\wh R_t$. Then in the Brownian case
  the reflection boundary jumps on average ``earlier'' and ``higher'' than in
  the Ornstein-Uhlenbeck case. Since both $X_t - R_t$ and $\wh X_t- \wh R_t$ are
  ``shortened'' at rate proportional to their values the described effect is not
  very pronounced and the speed of both ratchets is comparable in this region. 

  On the other hand if $\mu$ is large and $X_t - R_t \approx \wh X_t- \wh R_t$
  are close to zero then $\wh R_t$ has a higher chance to jump ``earlier'' and
  ``higher'' than $R_t$ because the drift $X_t$ towards $R_t$ is constant and
  that of $\wh X_t$ towards $\wh R_t$ is proportional to their distance which is
  small in this case. 

  Note that the above heuristic arguments are similar in spirit to the
  following considerations in the case without jumping reflection boundaries
  ($\gamma=0$, in that case the speed of both ratchets is zero).
  The invariant density of the reflected Brownian motion with negative drift
  $-\mu$ is $f(x) = 2\mu e^{-2\mu x}$, $x\ge 0$ \citep[see e.g.\ ][p.~94]{0659.60112} 
  and that of the reflected Ornstein-Uhlenbeck process with drift $-\mu x$ and
  unite variance is $g(x) = 2\sqrt{\mu/\pi} e^{-\mu x^2}$, $x \ge 0$ (this can
  be easily obtained from the invariant density of the Ornstein-Uhlenbeck
  process). For the expectations we have $\int_0^\infty x f(x)\, dx = 1/(2\mu)$
  and $\int_0^\infty x g(x)\, dx = 1/\sqrt{\pi \mu}$. In particular, under the
  invariant distributions the expectation of the reflected Brownian motion is
  larger than that of the reflected Ornstein-Uhlenbeck process for $\mu<\pi/4$
  whereas for $\mu>\pi/4$ the opposite inequality holds.\qed  
\end{remark}

\subsubsection*{Outline} 
\label{sec:outline}
The rest of the paper is split in two Sections in which Theorem~\ref{theoremRBM}
and Theorem~\ref{theoremOU} are proved. In Section~\ref{sec:gr-constr} we deal
with the Brownian ratchet. First, in Section~\ref{sec:rbmd} we recall an
explicit construction of the reflecting Brownian motion with drift. It will be
used in Section~\ref{sec:gr-constr-app} to give a graphical construction of the
Brownian ratchet. There we also prove a scaling property for the Brownian
ratchet and show that the graphical construction can also be used to construct a
coupling of Brownian ratchets with different initial conditions. Between the
jump times the Brownian ratchet can be seen as a killed reflecting Brownian
motion with drift. For that reason in Section~\ref{sec:green} we compute the
corresponding Green function and obtain several estimates on the moments of the
killing time and the position at killing time. In Section~\ref{sec:invdist} we
study the Markov chain of the increments of the Brownian ratchet at jump times.
We show that this Markov chain possesses a unique invariant distribution and
compute the expectations under this distribution. These will be used later to
compute the speed of the ratchet explicitly. In Section~\ref{sec:reg-brr} we
define a regeneration structure for the Brownian ratchet and show that the
increment at these regeneration times have finite second moments. From that we
obtain in Section~\ref{sec:proof12} the assertion of Theorem~\ref{theoremRBM}.
In Section~\ref{sec:OU-ratchet}, which has a similar structure to
Section~\ref{sec:gr-constr}, we carry out the corresponding program for the
Ornstein-Uhlenbeck ratchet.

\section{Brownian ratchet with negative local drift} 
\label{sec:gr-constr}

In this section we give a graphical construction of the Brownian ratchet with
negative local drift from which we deduce a scaling property and show that the
construction allows to couple two Brownian ratchets so that from some almost
surely finite time on they have the same spatial as well as temporal increments.
Then we study the Markov chain of the increments of the ratchet at the jump times
of the boundary and show that it has a unique invariant distribution, which will
allow to compute the speed of the ratchet explicitly. For the LLN and CLT we
define regeneration times of the ratchet and show that the increments between
these times have bounded second moments.

Before we start with the above schedule let us recall the definition and an
explicit construction of the reflecting Brownian motion with drift.

\subsection{Reflecting Brownian motion with drift}
\label{sec:rbmd}

Though the definition given here is valid for any $\mu \in \R$ we will assume
$\mu \ge 0$ because the case $\mu < 0$ is less interesting. For more information
on reflecting Brownian motion with drift we refer to e.g.\
\citep{0659.60112,0965.60077,Peskir:2006}.
A \emph{reflecting Brownian motion} with infinitesimal drift $-\mu$ started in
$x \ge 0$, which we denote by $RBM^x(-\mu)$, is a strong Markov process with
continuous paths (i.e.\ a diffusion process) associated with the infinitesimal
operator $A^\mu$ acting on
\begin{align*}
  \mathcal D(A^\mu)\coloneqq  \{f \in C_b^2(\R_+) : f'(0+) =0\}
\end{align*}
as follows: 
\begin{align}
  \label{eq:gener}
  A^\mu f(y) : =  \frac 12 f''(y) -\mu f'(y).   
\end{align}
We shall omit the superscript $x$ and write $RBM(-\mu)$ whenever the initial
value is not important. 

Let us also recall from \citep{Peskir:2006} an explicit construction of 
$RBM^x(-\mu)$ that will be useful for our purposes. Let $B=(B_t)_{t\ge 0}$ be
a standard Brownian motion starting in $0$. We define the Brownian motion  
with drift $\mu$, denoted by $B^\mu$, and its running maximum, denoted by
$M^\mu$, by 
\begin{align}
  \label{eq:bm.max}
  B^\mu_t = B_t+ \mu t \quad  \text{and} \quad M_t^\mu = \max_{0\le s \le t}
  B_s^\mu, \quad \text{for }  t \ge 0. 
\end{align}
Furthermore we define $Z^{\mu,x}=(Z_t^{\mu,x})_{t\ge 0}$ by   
\begin{align}
  \label{eq:RBM}
  Z_t^{\mu,x} = (x  \vee M^\mu_t) - B^\mu_t. 
\end{align}
Then in \citep[][Thm.~2.1]{Peskir:2006} it is shown that
\begin{align}
  \label{eq:laweq}
  RBM^x(-\mu) \stackrel{d}{=} Z^{\mu,x}.  
\end{align}

\subsection[Graphical construction of the Brownian ratchet]{Graphical
  construction of  the Brownian ratchet  with negative  local drift}  
\label{sec:gr-constr-app}

  Assume $\mu\ge 0$ and let $B^\mu$ be as in \eqref{eq:bm.max}. Furthermore let
  $N^\gamma$ be an independent Poisson process on $\R \times [0,\infty)$ with
  intensity $\gamma \lambda^2(dx,dt)$ where $\lambda^2$ is Lebesgue measure on
  $\R^2$.

  We define a sequence of jump times $(\tau_n)_{n=0,1,\dots}$ and a sequence
  $(S^{(n)})_{n=0,1,\dots}$ with $S^{(n)} = (S_t^{(n)})_{t \ge \tau_n}$ as
  follows: 
  \begin{align}
    \label{eq:29} 
    \tau_0 & \coloneqq 0, \\ 
    \label{eq:18} 
    S_0^{(0)}& \coloneqq x_0, \quad  S_t^{(0)}  \coloneqq \max\bigl\{S_0^{(0)},\sup_{0 \le s \le
      t} \{B_s^\mu\} \bigr\}. 
  \end{align}
  Given $\tau_{n-1}$ and $S^{(n-1)}$ for some  $n \ge 1$ we set  
  \begin{align} 
    \label{eq:19}  
    \tau_n &\coloneqq \inf\bigl\{ t >\tau_{n-1}: N^\gamma \cap [B_t^\mu,S_t^{(n-1)}]\times \{t\} 
    \ne  \emptyset\bigr\}.     
  \end{align}
  Furthermore we let $S_{\tau_n}^{(n)}$ be the space component of the almost surely unique 
  element of $ N^\gamma \cap [B_{\tau_{n}}^\mu,S_{\tau_n}^{(n-1)}]\times
  \{\tau_n\}$. For $t \ge \tau_n$ define   
  \begin{align*}
    S_t^{(n)} & \coloneqq \max\bigl\{S_{\tau_n}^{(n)}, \sup_{\tau_n \le s\le t}
    \{B_s^\mu\}\bigr\}.  
  \end{align*}
  
  Finally we define $S \coloneqq (S_t)_{t\ge 0}$  and $(\XX,\RR)\coloneqq (X_t,R_t)_{t \ge 0} $ by
  setting   
  \begin{align}
    \label{eq:def.SXR}
    \begin{split}
      S_t & =S_t^{(n)} \quad \text{and} \quad  R_t=\sum_{i=1}^n 
      \left(S_{\tau_i}^{(i-1)}-S_{\tau_i}^{(i)}\right) \; \; \text{for $t \in
        [\tau_n,\tau_{n+1})$} \\ 
      \text{and }  \quad  
      X_t & =R_t + S_t-B_t^\mu \quad  \text{for $t \ge 0$.}  
    \end{split}
  \end{align}
  Note that $S$ is the ``running maximum'' process that jumps down to Poisson points
  that are between the process itself and the Brownian motion with drift. 

\begin{figure}[htb]
  \begin{center}
    \includegraphics[width=0.98\textwidth]{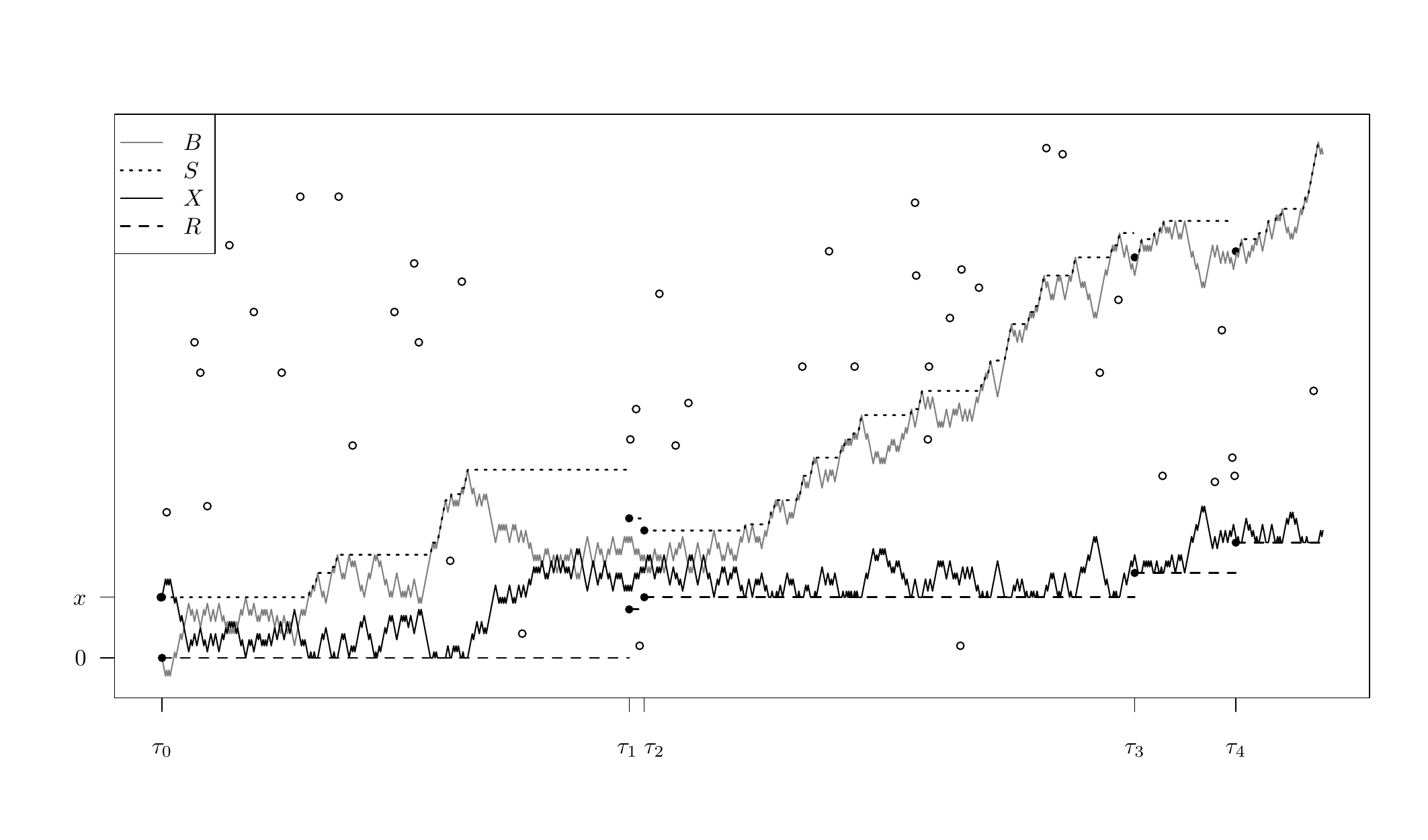}%  
    \caption{\label{fig:br-ratchet} Graphical construction of the Brownian
      ratchet with locally negative drift.}
  \end{center}
\end{figure}

In the following lemma we verify that $(\mathcal X,\mathcal R)$ fits the
description of $(\gamma,\mu)$-Brownian ratchet given in
Subsection~\ref{sec:model}. 
\begin{lemma} \label{lem:gr-constr-BR}
  The process $(\mathcal X,\mathcal R)$ is $(\gamma,\mu)$-Brownian ratchet 
  started in $(x,0)$.  
\end{lemma}
\begin{proof}
  By construction $S_{\tau_i}^{(i-1)} \ge S_{\tau_i}^{(i)}$, so that $\mathcal
  R$ is non-decreasing. Furthermore, $S_t \ge B_t$ implies $ X_t \ge  R_t$ for
  all $t \ge 0$. Between $\tau_n$ and $\tau_{n+1}$  the process  
  $\mathcal X$ is $RBM(-\mu)$ reflected at $R_t$ starting at time $\tau_n$ in  
  \begin{align*}
    R_{ \tau_n} + S_{\tau_n}-B^\mu_{\tau_n} &= \sum_{i=1}^n (S_{
      \tau_i}^{(i-1)} - S_{ \tau_i}^{(i)}) + S_{\tau_n}^{(n)} -
    B^\mu_{\tau_n} \\ &   
    = \sum_{i=1}^{n-1} (S_{ \tau_i}^{(i-1)} - S_{\tau_i}^{(i)}) + S_{
      \tau_n}^{(n-1)} - B^\mu_{\tau_n} = R_{\tau_{n-1}} +
    S_{\tau_n-}-B^\mu_{\tau_n}.   
  \end{align*}
  Thus, $X_{\tau_n} =  X_{\tau_n-}$ for all $n \ge 0$, and therefore
  the paths of $\mathcal X$ are continuous.    

  Given the process up to time $\tau_n$ the jump rate of $\mathcal R$ i.e.\ the
  rate of at which $\tau_{n+1}$ occurs is $\gamma(S_t-B_t^\mu)=\gamma( X_t-  R_t)$.
  Then the reflection boundary jumps to $R_{\tau_{n+1}}=
  R_{\tau_n}+S_{ \tau_{n+1}}^{(n)} - S_{ \tau_{n+1}}^{(n+1)}$. By homogeneity of the
  Poisson process $N^\gamma$, $S_{\tau_{n+1}}^{(n+1)}$ is uniform on
  $[B^\mu_{\tau_{n+1}},S_{\tau_{n+1}}^{(n)}]$. Thus, $R_{\tau_{n+1}}$ is uniform on 
  $[R_{\tau_n}, X_{\tau_{n+1}}]$ because for some $U \sim U([0,1])$ we have 
  \begin{align*}
    R_{\tau_{n+1}} & = R_{\tau_n}+S_{\tau_{n+1}}^{(n)} - (B_{\tau_{n+1}}^\mu + U
    (S_{\tau_{n+1}}^{(n)}-B_{\tau_{n+1}}^\mu))\\
    & = R_{\tau_n}+ (1-U)(S_{ \tau_{n+1}}^{(n)} - B_{\tau_{n+1}}^\mu) =
    R_{\tau_n}+ (1-U)( X_{ \tau_{n+1}}- R_{\tau_n}).
  \end{align*}
\end{proof}

If we transform time and space in the graphical construction then we of course  
rescale the Brownian motion with drift and transform the Poisson process. There
is only one such transformation that maps the Poisson process $N^1$ to
$N^\gamma$ and the Brownian motion with drift to a Brownian motion with another
drift. 
\begin{lemma}[Scaling property] 
  \label{lem:scaling}
  For $\gamma>0$ and $\mu \ge 0$ we have  
  \begin{align}
    \label{eq:br.scaling}
    (X_t^{\gamma,\mu},R_t^{\gamma,\mu})_{t\ge 0} \stackrel{d}{=}\gamma^{-1/3}
    (X_{\gamma^{2/3}t}^{1,\gamma^{-1/3}\mu},R_{\gamma^{2/3}t}^{1,\gamma^{-1/3}\mu})_{t\ge 
      0}. 
  \end{align}
\end{lemma}
\begin{proof}
  Assume that we construct
  $(X_{t}^{1,\gamma^{-1/3}\mu},R_{t}^{1,\gamma^{-1/3}\mu})_{t\ge 0}$ starting
  in $(\gamma^{1/3}x_0,0)$ using the Poisson process $N^1$ and the Brownian motion
  with drift $B^{\gamma^{-1/3}\mu}$.  We define $g:\R^2 \to \R^2 $ by
  \begin{align*}
    g(t,x) = (\gamma^{-2/3}t,\gamma^{-1/3}x).  
  \end{align*}
  On the one hand, rescaling space and time using $g$ we obtain the process on
  the right hand side of \eqref{eq:br.scaling}. On the other hand the rescaled
  graphical construction leads to the process on the left hand side of
  \eqref{eq:br.scaling}. For that we need to verify
  \begin{align} 
    \label{eq:13} 
    g(N^1) & \stackrel d = N^\gamma,\\ 
    \label{eq:21} 
    g \bigl((t,B_{t}^{\mu\gamma^{-1/3}})_{t\geq 0}\bigr) & \stackrel d =
    \bigl(t,B_{t}^{\mu}\bigr)_{t\geq 0}.  
  \end{align}
  Equation \eqref{eq:13} is clear. For \eqref{eq:21} we have using the scaling
  property of the Brownian motion
  \begin{align}
    \label{eq:14}
    \begin{split}
      g \bigl((t,B_{t}^{\mu\gamma^{-1/3}})_{t\geq 0}\bigr) & = 
      \bigl(\gamma^{-2/3}t,\gamma^{-1/3}B_{t}^{\mu\gamma^{-1/3}}\bigr)_{t\geq
        0} = \bigl(\gamma^{-2/3}t, \mu \gamma^{-2/3} t +
      \gamma^{-1/3}B_{t}^{0})\bigr)_{t\geq 0}  \\
      & \stackrel{d}{=}\bigl(\gamma^{-2/3}t,\mu \gamma^{-2/3}t +
      B_{\gamma^{-2/3}t}^{0}\bigr)_{t\geq 0}  =
      \bigl(t,B_{t}^{\mu}\bigr)_{t\geq 0}.  
    \end{split}
  \end{align}
\end{proof}

We now turn to the construction of a coupling of two $(\gamma,\mu)$-Brownian
ratchets starting in $\bigl((x,0),(\wt x,0)\bigr)$. Let $(B_t^\mu)_{t\ge0}$
and $N^\gamma$ be as before, and let $x,\wt x \ge 0$ with $x \ge \wt x$
without loss of generality. To construct the coupled Brownian ratchet 
\begin{align}
  \label{eq:coupl}
  \bigl((X_t,R_t),(\widetilde X_t,\widetilde R_t)\bigr)_{t\ge   0}, 
\end{align}
set $\tau_0=\wt\tau_0=0$, $S_0^{(0)}=x$, $\widetilde S_0^{(0)}=\wt x$ and
define as before the sequences $(\tau_n)_{n\ge 0}$, $(\wt \tau_n)_{n \ge 0}$,
$(S^{(n)})_{n\ge 0}$ and $(\wt S^{(n)})_{n\ge 0}$. Furthermore define the
corresponding processes $S$, $\wt S$, $(X_t,R_t)_{t\ge 0}$ and $(\widetilde
X_t,\widetilde R_t)_{t\ge   0}$ as in \eqref{eq:def.SXR}. 
Define the coupling time by  
\begin{align} 
  \label{eq:def.Tcoupl} 
  T_{\textnormal{coupl}}=\inf\{t \ge 0: S_t = \widetilde S_t\}.   
\end{align}
Note that, since we use the same Brownian motion and the same Poisson 
process for both ratchets we have $S_t = \widetilde S_t$ for all  $t \ge
T_{\textnormal{coupl}}$. Thus, on the event $\{T_{\textnormal{coupl}} < \infty\}$,
there are $n, \wt n$ such that for $k \ge 0$ we have almost surely 
\begin{align} 
  \label{eq:coupl.cons}
  \begin{split}
    \tau_{n+k} & =\wt \tau_{\wt n +k}, \\
    X_{\tau_{n+k+1}}-X_{\tau_{n+k}} & =\widetilde X_{\tau_{n+k+1}}-\widetilde
    X_{\tau_{n+k}},   \\ 
    R_{\tau_{n+k+1}}-R_{\tau_{n+k}} & =\widetilde R_{\tau_{n+k+1}}-\widetilde
    R_{\tau_{n+k}}. 
  \end{split}
\end{align}
The following lemma shows that $T_{\textnormal{coupl}}$ is almost surely finite,
i.e., the coupling is successful.    
\begin{lemma}[Exponential moments of the coupling time]
  \label{lem:coupling} 
  \leavevmode%
  \\ 
  For $\mu >0$ and $0 \le \alpha \le \mu^2/2$ we have   
  \begin{align*}
    \mathbb E [e^{\alpha T_{\textnormal{coupl}}}] \le e^{x(\mu
      -\sqrt{\mu^2-2\alpha })}.    
  \end{align*}
\end{lemma}
\begin{proof}
  At time $T_{\textnormal{coupl}}$ either both $S$ and $\widetilde S$ use the
  same point of the Poisson process $N^\gamma$ or the Brownian motion $B^\mu$
  touches  the maximum of $S$ and $\widetilde S$ which is $S$ by assumption 
  $\wt x \le x$. Thus, we have $T_{\textnormal{coupl}}\le T$ for   
  \begin{align*}
    T\coloneqq \inf\{t > 0: S_t=M_t\} = \inf\{t > 0: S_t=B_t^\mu\}.  
  \end{align*}
  By construction $S_t$ can increase only after this time $T$ and decrease by
  jumping down when it uses points of the Poisson process. Ignoring this
  decrease by jumping down we obtain  
  \begin{align*}
    T \le H_x\coloneqq  \inf\{t >0 : B_{t}^\mu = x\}.   
  \end{align*}
  It is well known that \citep[see e.g.][p.295]{BorodinSalminen:2002} for
  $\alpha \le \mu^2/2$ we have  
  \begin{align*}
    \mathbb E [e^{\alpha H_x}] = e^{x(\mu -\sqrt{\mu^2 - 2\alpha})}
  \end{align*}
  and the result follows. 
\end{proof}

\subsection[Green function]{Green function of the killed reflected Brownian
  motion with drift} 
\label{sec:green}

Since between the jumps the ratchet constructed in the previous section behaves
as a killed reflected Brownian motion we will need in the sequel some
functionals of that process, such as expected killing time or expected position
at killing. To this end we need to compute the corresponding Green function.

Let us first give a short description on how the Green function of a killed 
diffusion can be computed; for details we refer to Chapter~II in
\citep[][]{BorodinSalminen:2002} or Chapter~4 in \citep{MR0345224}. 

\begin{remark}[Green function of a reflected diffusion with
  killing] \label{rem:Green-gener} \leavevmode%
  \\
  Let $Y:=(Y(t))_{t \ge 0}$ be a reflected diffusion process with killing on
  state-space $[0,\infty)$ associated with infinitesimal operator $A$ acting on 
\begin{align*}
  \mathcal D(A)\coloneqq  \{f \in C_b^2(\R_+) : f'(0+) =0\} 
\end{align*}
as follows
\begin{align}
  \label{eq:generY}
  A f(y) \coloneqq \frac 12 f''(y) +b(x) f'(y) - c(x) f(x). 
\end{align}
We consider in this paper the cases $b(x) = -\mu$ or $b(x) = -\mu x$ and
$c(x)=\gamma x$. Since in both cases the killing time is almost surely finite
the resulting diffusions are transient. The speed and the killing measures of $Y$ 
are given by \citep[see e.g.][p.~17]{BorodinSalminen:2002}
\begin{align} \label{eq:speed-kill-Y}   
  m(dx)=m(x) \, dx \coloneqq  2e^{B(x)}\, dx  \quad \text{and} \quad  k(dx)=k(x)
  \, dx \coloneqq  2 c(x) e^{B(x)}\, dx,   
\end{align}
where $B(x) \coloneqq \int_0^x 2 b(y) \, dy$. Let $p(\cdot;\cdot,\cdot)$ denote 
the transition density of $Y$ with respect to the speed measure. Then the Green
function of $Y$ is defined by      
\begin{align*}
  G(x,y)\coloneqq \int_0^\infty p(t;x,y)\, dt. 
\end{align*}
In the transient regular case (the latter means here that every point in
$[0,\infty)$ can be reached with positive probability starting from any other
point) the Green function of $Y$ is positive and finite. It is obtained in terms
of two independent solutions $\phi$ and $\psi$ of the differential equation
$Af=0$ that are both unique up to a constant factor and satisfy the following
conditions
\begin{itemize}
\item[(i)] $\phi$ is positive and strictly decreasing with $\phi(x) \to 0$ as $x
  \to \infty$, 
\item[(ii)] $\psi$ is positive and strictly increasing,
\item[(iii)] $\psi'(0+) = 0$ (this condition is for the reflecting boundary).  
\end{itemize} 
The Wronskian, defined by $w(\psi,\phi) \coloneqq \psi'(x)\phi(x) -
\psi(x)\psi'(x)$ is independent of $x$. Thus, the functions $\phi$ and 
$\psi$ can be chosen so that their Wronskian equals one. Then, the Green
function of $Y$ is given by
\begin{align} \label{eq:Green-def} 
  G(x,y) = 
  \begin{cases}
    \phi(x)\psi(y) & : \; 0 \le y \le x,\\ 
    \psi(x) \phi(y) & : \; 0 \le x \le y.      
  \end{cases} 
\end{align}
In our computations we will in principle not need the exact expressions for
$\phi$, $\psi$ and $G$. In particular in the case of the Ornstein-Uhlenbeck  
process, where these solutions depend in a complicated manner on the model
parameters, we will not compute the function $\psi$ explicitly. Asymptotic
bounds at infinity will suffice for our purposes.    
\end{remark}

In the following remark we collect some properties of the Airy functions that
will be needed in the sequel. For further properties we refer to
\citep{AbramowitzStegun:1972} (cf. also Remark~5.2 in
\citep{DepperschmidtPfaffelhuber:2010}).

\begin{remark}[Airy functions]
  \label{rem:airy}   
  The Airy functions $Ai$ and $Bi$ are two linearly independent solutions of the
  differential equation
  \begin{align}
    \label{eq:4new}
    u''(x) - x u(x)=0. 
  \end{align}
  We will only need the properties of the Airy functions on $[0,\infty)$. On
  that domain the functions are positive, $Ai$ is decreasing with
  $Ai(x)\xrightarrow{x\to\infty} 0$ and $Bi$ is increasing with
  $Bi(x)\xrightarrow{x\to\infty} \infty$.

  The Wronskian is independent of $x$ and is given by 
  \begin{align} 
    \label{eq:56} 
    w(Ai,Bi) = %\coloneqq % Bi'(x)Ai(x)-Ai'(x)Bi(x) = 
    Bi'(0)Ai(0)-Ai'(0)Bi(0)  = \frac1\pi.    
  \end{align}
  The integral of $Ai$ on  $\R_+$ is  
  \begin{align} \label{eq:59} 
    \int_0^\infty Ai(u)\,du = \frac13.
  \end{align} 
  We will also need the function 
  \begin{align} \label{eq:58}
    Gi(x) \coloneqq  Ai(x)\int_0^xBi(y)\,dy + Bi(x)\int_x^\infty Ai(y)\,dy. 
  \end{align}
  For fixed $\mu \ge 0$ and $C \in \R$ we define a function $M$ by    
  \begin{align}  
    \label{eq:def.M} 
    \begin{split}
      M(x) & \coloneqq  \pi\Bigl\{Ai(\mu^2+x) \int_0^x (Bi(\mu^2+y) +C Ai(\mu^2+y)) \, dy \\ &  
      \phantom{AiBi} + (Bi(\mu^2+x) +C Ai(\mu^2+x))\int_x^\infty  Ai(\mu^2+y) \,
      dy\Bigr\}.   
    \end{split}
  \end{align}
  Using \eqref{eq:59}, \eqref{eq:58} and positivity of $Ai$ and $Bi$ on
  $[0,\infty)$ we have for $x \ge 0$
  \begin{align} 
    \label{eq:bound.M}  
    M(x) \le \pi\left( Gi(\mu^2+x) + \frac{\abs{C}}3 Ai(\mu^2+x)\right).    
  \end{align}
  Since the functions $Gi$ and $Ai$ are bounded on $[0,\infty)$
  we may define    
  \begin{align}
    \label{eq:g-star}
    G^*\coloneqq  G^*(C,\mu)\coloneqq  \pi \cdot \max_{x \ge 0} \Bigl\{
    Gi(\mu^2+x) +  \frac{\abs{C}}3 Ai(\mu^2+x) \Bigr\}.     
  \end{align}
  \hfill \qed
\end{remark}

\begin{definition}[Killed reflecting Brownian motion with drift]   
  \label{rem:killed-rbm} 
  Let $ Z \coloneqq (Z_t)_{t\ge 0}$ denote a reflecting Brownian motion with  
  drift $-\mu$ starting in $x \ge 0$ (see \eqref{eq:RBM}) under the law $\Pr_x$
  and let $\E_x$ denote the corresponding expectation. Furthermore using an
  exponentially distributed rate $1$ random variable $\xi$ independent of $Z$ we
  define the \emph{killing time} by 
  \begin{align*}
    \tau=\inf\{t >0: \gamma \int_0^t  Z_s\,ds \ge \xi\}.   
  \end{align*}
  Then \emph{reflecting Brownian motion with infinitesimal drift $-\mu$ killed
    at rate $\gamma Z$} is defined as the process $Z^K \coloneqq (Z^K_t)_{t\ge 0}$ 
  with $ Z^K_t =  Z_t$ for $t \in [0,\tau)$ and $ Z_t^K = \Delta$ for $t \ge
  \tau$ for some $\Delta \not\in \R$, often referred to as the \emph{cemetery 
    state.} The infinitesimal operator $A^{\mu,\gamma}$ corresponding to $
  Z^K$ acts on $C^2$ functions $f:[0,\infty) \to \R$ satisfying $f'(0+)=0$ as 
  follows  
  \begin{align} 
    \label{eq:gen-killed-RBM} 
     A^{\mu,\gamma} (x) = \frac12 f''(x)  - \mu f'(x) -\gamma x f(x).       
  \end{align}
\end{definition}
In view of the scaling property it is enough to prove the results for the
$(\gamma,\mu)$-Brownian ratchet in a particular case. In what follows we assume
$\gamma = \frac12$ and fix $\mu \ge 0$. In this case the speed and killing
measure corresponding to $Z^K$ are given by
%\citep[e.g.][p.~17]{BorodinSalminen:2002}
\begin{align} \label{eq:speed-kill-KRBM}   
  m(dx)=m(x) \, dx \coloneqq  2e^{-2\mu x}\, dx  \quad \text{and} \quad  k(dx)=k(x)
  \, dx \coloneqq  x e^{-2\mu x}\, dx.
\end{align}

\begin{lemma}[Green function of killed reflecting Brownian motion with
  drift]\leavevmode \\  
  The \label{lem:green-krbm} Green function of $ Z^K$ is given by  
  \begin{align} \label{eq:green-1}
    G(x,y)\coloneqq  
    \begin{cases}
      \phi(x)\psi(y) & : \; 0 \le y \le x,\\ 
      \psi(x)\phi(y) & : \; 0 \le x \le y,   
    \end{cases}
  \end{align}
  where %$w=1/\pi$ (as defined in \eqref{eq:56}) and  
  $\phi, \psi:  [0,\infty) \to \R$ are defined by  
  \begin{align} \label{eq:def-phi-psi} 
    \begin{split}
      \phi(x) & = \pi e^{\mu x} Ai(\mu^2 + x), \\
      \psi(x) & = C  e^{\mu x} Ai(\mu^2 + x) + e^{\mu x} Bi(\mu^2 +x),    
    \end{split}
  \end{align}
  with 
  \begin{align} 
    \label{eq:Cmu}
    C\coloneqq C(\mu)\coloneqq -\frac{\mu Bi(\mu^2)+Bi'(\mu^2)}{\mu Ai(\mu^2)+Ai'(\mu^2)}.   
  \end{align} 
\end{lemma}
\begin{proof}
  As explained in Remark~\ref{rem:Green-gener} the Green function $G$ is
  obtained in terms of solutions of 
  \begin{align}
    \label{eq:green-2}
     A^{\mu,1/2} u(x) = 0, \; x  \ge 0,    
  \end{align}
  where $ A^{\mu,1/2}$ is defined in \eqref{eq:gen-killed-RBM}.  
  
  The functions $\phi$ and $\psi$ defined in \eqref{eq:def-phi-psi} are two
  independent solutions of \eqref{eq:green-2}. In Lemma~\ref{lem:prop-phi-psi}
  we show that they satisfy conditions (i) and (ii) from
  Remark~\ref{rem:Green-gener}, whereas (iii) holds by the choice of $C=C(\mu)$.
  It remains to show $w(\psi,\phi)=1$. Using independence of the Wronskian of
  $x$ and \eqref{eq:56} we obtain
  \begin{align*}
    \psi'(0)\phi(0)-\psi(0)\phi'(0) & = \mu
    \psi(0)\phi(0)+(Bi'(\mu^2)+CAi'(\mu^2)) \pi Ai(\mu^2) 
    \\ & \qquad
    - \mu \phi(0)\psi(0) - \pi Ai'(\mu^2) (Bi(\mu^2)+CAi(\mu^2)) \\   
    & = \pi\left (Bi'(\mu^2)Ai(\mu^2)-Bi(\mu^2)Ai'(\mu^2)\right)   \\
    & = \pi\left(Bi'(0)Ai(0)-Bi(0)Ai'(0)\right) =1.
  \end{align*}
  In particular, $w(\psi,\phi)$ is also independent of $\mu$. Altogether the
  assertion of the lemma follows.    
\end{proof}

%\bigskip 

\begin{lemma}[Properties of $\phi$ and $\psi$] % \leavevmode \\  
  Let $\phi$ and $\psi$ be defined by \eqref{eq:def-phi-psi}.
  For \label{lem:prop-phi-psi} any $\mu \ge 0$ the function $\phi$ is strictly
  decreasing and the function $\psi$ is strictly increasing in $x$.
\end{lemma}
\begin{proof}
  Properties of the Airy function $Ai$ (see Remark~\ref{rem:airy}) imply that
  for any $\mu \ge 0$ the function $\phi$ is positive and that $\phi(x)\to 0$
  as $x\to \infty$. We will show that $\phi'(x) <0$. To this end, it is enough
  to show that for $\mu \ge 0$, $x \ge 0$ 
  \begin{align*}
    g(\mu,x) \coloneqq  \mu Ai(\mu^2+x)+Ai'(\mu^2+x) < 0.   
  \end{align*}

  First we show that $g_1(\mu)=g(\mu,0) <0$ for $\mu\ge0$. The assertion is
  true for $\mu=0$ and for $\mu \to \infty$ we have $g_1(\mu)\to 0$. So if
  $g_1(\mu)$ is positive on some interval then there is a local maximum in
  some $\mu_0$ such that on the one hand we have $g_1(\mu_0) >0$ and on the
  other hand 
  \begin{align*}
    0  & = g_1'(\mu_0)  = Ai(\mu_0^2) +  
    2\mu_0^2 Ai'(\mu_0^2) + 2\mu_0 Ai''(\mu_0^2) \\  & = Ai(\mu_0^2) +
    2\mu_0^2(Ai'(\mu_0^2) + \mu_0 Ai(\mu_0^2)) \\ 
    & = Ai(\mu_0^2) + 2\mu_0^2 g_1(\mu_0) >0,
  \end{align*}
  leading to a contradiction. 
  
  Now we fix $\mu \ge 0$ and show $g(\mu,x)<0$ for all $x \ge 0$. For $x=0$ it
  is true by the above argument. As $x \to \infty$ we have $g(\mu,x) \to
  0$. If $g(\mu,\cdot)$ has positive values in the interval $(0,\infty)$ then 
  there is a local maximum $x_0$ such that $g(\mu,x_0) >0$ and  
  \begin{align*}
    0 & = \frac{\partial}{\partial x}
    g(\mu,x_0)  = \mu Ai'(\mu^2+x_0) + Ai''(\mu^2+x_0) \\ 
    & = \mu Ai'(\mu^2+x_0) + \mu^2Ai(\mu^2+x_0) + x^2 Ai(\mu^2+x_0)\\ &  
    = \mu  g(\mu,x_0)+x^2Ai(\mu^2+x_0) >0,       
  \end{align*}
  leading again to a contradiction. 
  
  It remains to show that $\psi$ is increasing. By the choice of $C$ we have
  $\psi'(0)=0$. Let  $h(\mu,x)=\mu Bi(\mu^2+x)+Bi'(\mu^2+x)$. For all $x>0$ and
  $\mu \ge 0$ we have  
  \begin{align*}
    \psi'(x)= e^{\mu x}\bigl(h(\mu,x) -
    \frac{h(\mu,0)}{g(\mu,0)}g(\mu,x)\bigr) >0. 
  \end{align*}
  To see this note that, as we have shown above,
  $g(\mu,x)/g(\mu,0)=|g(\mu,x)/g(\mu,0)|<1$ and therefore   
  \begin{align*}
    h(\mu,x) - \frac{h(\mu,0)}{g(\mu,0)}g(\mu,x) \ge h(\mu,x) - h(\mu,0) >0. 
  \end{align*}
  The last inequality follows from the fact that $Bi$ and $Bi'$ are
  increasing.   
\end{proof}

We set 
\begin{align} 
  \label{eq:tildephi}
  \Phi(x) & \coloneqq e^{-2\mu x} \phi(x) = \pi e^{-\mu x} Ai(\mu^2+x), \\    
  \label{eq:tildepsi} 
  \Psi(x) & \coloneqq e^{-2\mu x} \psi(x) = e^{-\mu x}(Bi(\mu^2+x)+C Ai(\mu^2+x))        
\end{align}
and note that $\Phi$ and $ \Psi$ solve the differential equation    
\begin{align} 
  \label{eq:5} 
  u''(x) + 2\mu u'(x) - x u(x) =0,   
\end{align}
and that $ \Phi$ is up to a constant factor the unique decreasing solution of
that equation satisfying $ \Phi(x) \to 0$ as $x \to \infty$. Furthermore we have       
\begin{align}
  \label{eq:2} 
   \Psi'(0) + 2 \mu  \Psi(0)   = \psi'(0) = 0
 \end{align}
 and simple calculation shows 
\begin{align} 
  \label{eq:3}
  \phi(x) \Psi'(x)   - \psi(x)  \Phi'(x) = w (\psi,\phi) = 1.   
  %Bi'(\mu^2+x) Ai(\mu^2+x)
  %-Bi(\mu^2+x) Ai'(\mu^2+x)=w.   
\end{align}

\begin{remark}[Expected killing time and the density of the killing position]
  \label{rem:exp-kill} \leavevmode%  
  \\ 
  Using the Green function one can compute the mean killing time of the killed  
  reflecting Brownian motion starting in $x \ge 0$. It is given by   
  \begin{align}\label{eq:1} 
    \E_x [\tau] & = \int_0^\infty G(x,y) m(y) \, dy = \int_0^\infty
    G(x,y)2e^{-2\mu y}  \, dy \\ 
    \intertext{which can be written as}  \label{eq:42} 
    & = 2\left(\phi(x) \int_0^x \Psi(y)\,dy + \psi(x) \int_x^\infty
      \Phi(y)\, dy \right).
  \end{align}
  Furthermore the density of $Z^K_{\tau-}$, i.e.\ the position at killing time
  is given by  \citep[see][p.~14]{BorodinSalminen:2002}   
  \begin{align} \label{eq:4} 
    G(x,y) k(y) = G(x,y) y e^{-2\mu y}.  
  \end{align} 
\qed
\end{remark}

\begin{lemma}[Exponential moments of the killing position]  
  \label{lem:exp.m.killing} \leavevmode \\    
  For $\alpha < \mu$ and  any $x \ge 0$  we have
  \begin{align}
    \label{eq:exp.m.killing}
    \E_x [e^{\alpha Z_{\tau-}^K}] < \infty.    
  \end{align}  
\end{lemma}
\begin{proof}
  Set $y^*\coloneqq \max_{y \ge 0}\{y e^{-(\mu-\alpha) y}\}$ and recall the
  function $M$ (for $C$ defined in \eqref{eq:Cmu}) and its bound $G^*$ in
  \eqref{eq:def.M} and \eqref{eq:g-star}. Then \eqref{eq:exp.m.killing} follows
  from
  \begin{align*}
    \E_x [e^{\alpha Z_{\tau-}}] & =  \int_0^\infty e^{\alpha y} G(x,y) y
    e^{-2\mu y} \, dy\\   
    & = \phi(x) \int_0^x y \psi(y)e^{-(2\mu-\alpha) y} \, dy + 
    \psi(x) \int_x^\infty y \phi(y)e^{-(2\mu-\alpha) y} \, dy \\ 
    & = e^{\mu x} \pi \Bigl\{Ai(\mu^2+x) \int_0^x y e^{-(\mu-\alpha)y} (Bi(\mu^2+y)
    +C Ai(\mu^2+y)) \, dy \\  
    & \qquad \quad + (Bi(\mu^2+x) +C Ai(\mu^2+x))\int_x^\infty y
    e^{-(\mu-\alpha)y} Ai(\mu^2+y) \, dy \Bigr\} \\ & \le e^{\mu x} y^* M (x)
    \le e^{\mu x} y^* G^*.   
  \end{align*}
\end{proof}

\begin{lemma}[Second moments of the killing time]  
  \label{lem:2nd.mom.kill.time}  \leavevmode \\     
  There is a positive finite constant $C^{\tn{kill}}$ such that for all $x \ge 0$  
  \begin{align}
    \label{eq:2nd.mom.kill.time}
    \E_x [\tau^2] < C^{\tn{kill}}.     
  \end{align}
\end{lemma}
\begin{proof}
  Since the killing time of the killed reflecting Brownian motion starting in $x
  \ge 0$ is bounded stochastically by the killing time of the killed reflecting
  Brownian motion starting in $0$, we have
  \begin{align*}
    \mathbb E_x[\tau^2] \le 1+ \mathbb E_0[\tau^2]. 
  \end{align*}
  By the Kac's moment formula \citep[see e.g.][ (5) on p. 119]{0962.60067} we
  have   
  \begin{align*}
    \E_0[\tau^2] & = 2 \int_0^\infty G(0,x) m(x) \int_0^\infty G(x,y) m(y) \, dy
    \, dx. \\ \intertext{Now using again \eqref{eq:def.M} and 
      \eqref{eq:g-star} as in the proof of Lemma~\ref{lem:exp.m.killing} we
      obtain}
    \E_0[\tau^2]  & \le 4 G^* \int_0^\infty G(0,x) m(x) e^{\mu x} \,
    dx % = 8 G^*\psi(0) \int_0^\infty \phi(x) \,dx 
= 8 G^*\pi \psi(0)
    \int_0^\infty Ai(\mu^2+x) \,dx . \\\intertext{Since $Ai$ is 
      decreasing we obtain}   
    \E_0[\tau^2] & \le 8 G^* \pi \psi(0) \int_0^\infty Ai(x) \, dx = \frac{8 G^*
    \pi}{3}\psi(0),   
  \end{align*}
  where the last equality follows from \eqref{eq:59}. 
\end{proof}

\begin{lemma}[Bound on the expected killing position starting from $x$]
  \label{lem:E_x<x+c} \leavevmode \\ 
  For any $x \ge 0$ 
  \begin{align}   
    \label{eq:41} 
    \mathbb E_x[Z_{\tau-}]\leq x + \phi(0) \psi(0) +\mu \E_0[\tau].    
  \end{align}   
\end{lemma}
\begin{proof}
  By \eqref{eq:4} and the definition of $\Phi$ and $\Psi$ in \eqref{eq:tildephi}
  respectively \eqref{eq:tildepsi} we have   
  \begin{align*}
    \mathbb E_x [Z_{\tau-}] & =  \int_0^\infty y^2 G(x,y)e^{-2\mu y} \,dy\\    
    &= \phi(x) \int_0^x y^2 \psi(y)  e^{-2\mu y} \,dy +  \psi(x) 
    \int_x^\infty  y^2 \phi(y)  e^{-2\mu y} \, dy \\ & = 
    \phi(x) \int_0^x y^2 \Psi(y)   \,dy +  \psi(x) \int_x^\infty  y^2 \Phi(y)
    \, dy.   
  \end{align*}
  Now using the fact that $\Phi$ and $\Psi$ satisfy \eqref{eq:5}, integration by
  parts, \eqref{eq:3} and the fact that $\Psi\phi = \psi\Phi$  we arrive at   
  \begin{align*}
    \mathbb E_x [Z_{\tau-}]  &= \phi(x) \int_0^x \left(y \Psi''(y)+ 2 \mu y\Psi'(y)\right) \, dy + 
    \psi(x) \int_x^\infty  \left(y \Phi''(y)+ 2 \mu y\Phi'(y)\right)\,
    dy \\  
    &= x \left[\phi(x)\Psi'(x)- \psi(x)\Phi'(x)\right]  \\  & \qquad \quad  +
    \phi(x) \left[-\Psi(x)+ \Psi(0) + 2 \mu x \Psi(x) - 2 \mu \int_0^x
      \Psi(y)\,dy\right]   
    \\ & \qquad \qquad + \psi(x) \left[\Phi(x) - 2 \mu x \Phi(x) - 2 \mu
    \int_x^\infty \Phi(y)\,dy\right] \\ 
    \\ &= x+ \phi(x) \Psi(0) - 2 \mu\left(\phi(x) \int_0^x 
      \Psi(y)\,dy + \psi(x) \int_x^\infty \Phi(y)\, dy \right)\\ 
    &= x+ \phi(x) \psi(0) - \mu \int_0^\infty G(x,y)m(y) dy \\ &= x+  \phi(x)
    \psi(0) - \mu \E_x[\tau].  
  \end{align*}
  Here the next to last equality follows from \eqref{eq:42} and from
  $\psi(0)=\Psi(0)$. Since $\phi$ is decreasing and $\E_x [\tau] \leq
  \E_0[\tau]$ the assertion \eqref{eq:41} follows.  
\end{proof}

\subsection{Invariant distribution at jump times} 
\label{sec:invdist}

In this subsection we consider the increments of the Brownian ratchet at the
jump times of the boundary process. We show that they constitute a Markov chain
with unique invariant distribution and compute the expected jump time and the
expected killing position under the invariant distribution. 
\begin{definition}[Markov chain at jump times] \label{eq:mkjt} 
  \leavevmode \\ 
  Let $(\mathcal X, \mathcal R)$ be $(\gamma,\mu)$-Brownian ratchet with 
  sequence of jump times of $\mathcal R$ given by $(\tau_n)_{n\ge 0}$. We define
  the Markov chain $(\mathcal Y,\mathcal W,\eta)
  \coloneqq (Y_n,W_n,\eta_n)_{n=1,2,\dots} $ of increments at jump times by     
  \begin{align}
    \label{eq:mc.inc}
    Y_n= X_{\tau_n}-R_{\tau_n}, && W_n = R_{\tau_n}-R_{\tau_{n-1}} && \text{and} &&
    \eta_n = \tau_n-\tau_{n-1}. 
  \end{align}
\end{definition}
\noindent 
Since for any $k$ the law of $(Y_n,W_n,\eta_n)_{n=k+1,k+2,\dots}$ depends on
$(Y_n,W_n,\eta_n)_{n = 1, \dots, k}$ only through $Y_k$, $(\mathcal Y,\mathcal
W,\eta)$ is indeed a Markov chain.

%The next result is an analogue of Proposition~5.6 in
%\citep{DepperschmidtPfaffelhuber:2010}.       
\begin{proposition} %[Existence and uniqueness of invariant distribution] 
  \label{prop:ex-uniq-inv}  %\leavevmode \\ 
  There exists a unique invariant distribution of the Markov chain $(\mathcal
  Y, \mathcal W, \eta)$. 
\end{proposition}
\begin{proof}
  While uniqueness of an invariant distribution is guaranteed by the coupling
  result in Lemma~\ref{lem:coupling}, to prove existence we need to show that
  the moments of $(Y_n,W_n,\eta_n)$ are bounded for all $n$. This implies then
  tightness of the sequence and also tightness of the Ces\`aro averages of the
  laws. Weak limits of subsequences of the latter are invariant distributions of
  the Markov chain. Boundedness of the moments of $\eta_n$ follows from
  Lemma~\ref{lem:2nd.mom.kill.time} and that of the moments of $Y_n$ (and $W_n$
  since it has the same distribution as $Y_n$) follows inductively from
  Lemma~\ref{lem:E_x<x+c}. For details we refer to the proof of Proposition~5.6
  in \citep{DepperschmidtPfaffelhuber:2010}.
\end{proof}

\begin{proposition} 
  \label{prop:exp-under-pi}
  Let $\nu$ be the invariant distribution of $Y_1,Y_2,\dots$ and let $\E_\nu$
  denote the expectation with respect to that distribution. Then there is a
  constant $K \in (0,\infty)$ so that
  \begin{align}
    \label{eq:22} 
    \E_\nu[Y_1] & = -\frac1K(\mu Ai(\mu^2)+Ai'(\mu^2)),\\  
    \intertext{and} \label{eq:36} 
    \E_\nu[\eta_1] & = \frac{ 2Ai(\mu^2)}K.     
  \end{align}
\end{proposition}
\begin{proof}
  Let $f_\nu$ denote the density of $\nu$ with respect to Lebesgue measure.
  Let $Z^{K,Y}:=(Z^{K,Y}_t)_{t \ge 0}$ be killed reflecting Brownian motion with
  drift starting in random value $Y \ge 0$ with increments independent of $Y$.
  Furthermore let $U$ be uniformly distributed on $(0,1)$. Invariance of $\nu$
  implies that for killing time $\tau$ of $Z^{K,Y}$ we have 
  \begin{align*}
    Y \stackrel{d}{=} U\cdot Z^{K,Y}_{\tau-}. 
  \end{align*}
  As in \citep[][Section~5.3]{DepperschmidtPfaffelhuber:2010} from that one can
  obtain the following recurrence equation 
  \begin{align*}
     f_\nu(z) & =  \int_0^\infty f_\nu(x) \int_z^\infty e^{-2\mu u} G(x,u)\, du 
  \end{align*}
  and then compute 
  \begin{align}
    \label{eq:7}
    f'_\nu(z) & = - \Psi(z) \int_z^\infty f_\nu(x) \phi(x) \, dx - \Phi (z)
    \int_0^z  f_\nu(x) \psi (x) \, dx, \\
    \label{eq:8}
    f''_\nu(z) & = - \Psi'(z) \int_z^\infty f_\nu(x) \phi(x) \, dx -  \Phi'(z) \int_0^z f_\nu(x) \psi(x)\,dx, \\
    \intertext{and}
    \label{eq:20}
    \begin{split}
      f'''_\nu(z) & = - \Psi''(z) \int_z^\infty f_\nu(x) \phi(x) \, dx +
      \psi'(z) f_\nu(z) \phi(z) \\ & \quad \qquad - \Phi(z) \int_0^z
      f_\nu(x)\psi(x) \, dx - \Phi(z) f_\nu(z) \psi(z) \\ & = - 2\mu
      f_\nu''(z) + z f'_\nu(z) + f_\nu(z),
    \end{split}
  \end{align}
  where for the last equality we used \eqref{eq:5}, equations for $f''_\nu$, 
  $f_\nu'$ and \eqref{eq:3}. Thus,   
  \begin{align*}
    f'''_\nu(z) =  - 2\mu f_\nu''(z) + (z f_\nu(z))'. 
  \end{align*}
  Integrating we obtain
  \begin{align} \label{eq:9}
    f''_\nu(z) =  - 2\mu f_\nu'(z) + z f_\nu(z).  
  \end{align}
  The integration constant is zero because from \eqref{eq:7}, \eqref{eq:8} and
  \eqref{eq:2} we see that $f''_\nu(0) = - 2\mu f'_\nu(0)$. By \eqref{eq:7} the
  density $f_\nu$ must be strictly decreasing.  Up to a constant factor the
  positive decreasing solution of \eqref{eq:9} is $ \Phi$ and
  it follows that 
  \begin{align}  
    \label{eq:inv-dens}
    f_\nu(z) = \frac1{K}  \Phi(z) \quad \text{with} \quad
    K=\int_0^\infty  \Phi(x)\, dx.
  \end{align} 

  From \eqref{eq:9} it follows 
  \begin{align*}
    \E_\nu[Y_1]= \int_0^\infty x f_\nu(x)\,dx & = \int_0^\infty (f''_\nu(x) + 2\mu 
    f'_\nu(x))\, dx \\ & = -f'_\nu(0)-2\mu f_\nu(0) = \frac\pi K(\mu 
    Ai(\mu^2)-Ai'(\mu^2)-2\mu Ai(\mu^2)) \\ 
    & = - \frac \pi K(\mu Ai(\mu^2) + Ai'(\mu^2)),   
  \end{align*}
  (replace here $K$ by $K\pi$ to get \eqref{eq:22}) and 
  \begin{align*}
    \E_\nu[\eta_1] & = 2\int_0^\infty f_\nu(x) \int_0^\infty e^{-2 \mu y}
    G(x,y)\, dy \, dx \\
    & = 2 \int_0^\infty f_\nu(x) \Bigl( \phi(x) \int_0^x \Psi(y)\,
    dy + \psi(x) \int_x^\infty  \Phi(y) \, dy\Bigl) \, dx.  \\
    \intertext{Now using Fubini's Theorem and then \eqref{eq:7} we have }
    \E_\nu[\eta_1] & = 2 \int_0^\infty \Bigl(\Psi(y) \int_y^\infty
    f_\nu(x)\phi(x) \, dx + \Phi(y) \int_0^y f_\nu(x)
    \psi(x) \, dx\Bigl) \, dy \\
    & = - 2 \int_0^\infty f'_\nu(x)\, dx = 2 f_\nu(0)= 2\frac{\Phi(0) }K =
    2\frac{\pi Ai(\mu^2)}{K}.   
  \end{align*}
  Again replacing $K$ by $K\pi$ we get to \eqref{eq:36}.   
\end{proof}

\subsection{Regeneration structure}  
\label{sec:reg-brr}

In this subsection we define a regeneration structure of the
$(\gamma,\mu)$-Brownian ratchet and show that the second moments of the
regeneration times and of the corresponding spatial increments are finite.

\begin{definition}[Brownian ratchet as a cumulative process] 
  \label{def:cumu1}   \leavevmode \\ 
  Given a Brownian ratchet $(\mathcal X, \mathcal R)=(X_t,R_t)_{t\ge 0}$ with
  $(X_0,R_0)=(x,0)$, $x \ge 0$ we define a sequence of regeneration times as
  follows 
  \begin{align} 
    \label{eq:15}
    \rho_0 & \coloneqq\inf\{t\ge 0: X_t=R_t\}, \qquad  \wt\rho_0 \coloneqq \inf\{t \ge
    \rho_0: R_{t-}\ne R_t\},
    \\ \intertext{where $R_{t-}=\lim_{s \to t,s<t} R_s$,
      and for $n \ge 1$ we set }   
    \label{eq:16} 
    \rho_{n} & \coloneqq\inf\{t \ge \wt\rho_{n-1}n: X_t=R_t\}, \qquad
    \wt\rho_{n} \coloneqq \inf\{t \ge \rho_{n}: R_{t-}\ne R_t\}.    
  \end{align}
  Then  $\rho_0 < \wt\rho_0 < \rho_1 < \wt\rho_1< \dots$ almost surely and
  we have $\rho_0=0$ in the case $x=0$. Furthermore the sequence 
  $(X_{\rho_{n+1}}-X_{\rho_{n}},\rho_{n+1}-\rho_{n})_{n\ge 0}$ is
  iid. We define  
  \begin{align} \label{eq:12} 
    M_t  \coloneqq  \min\{n : \rho_n >t\}, && S_n  \coloneqq  \sum_{i=1}^n
    (X_{\rho_i}-X_{\rho_{i-1}}) && \text{and} &&  A_t  \coloneqq 
    X_{\rho_0}+X_t-X_{\rho_{M_t}}.  
  \end{align}
  Then we have
  \begin{align}
    \label{eq:17}
    X_t = S_{M_t}+A_t,  
  \end{align}
 that is, $X_t$ is a type A cumulative process $S_{M_t}$ with 
 remainder $A_t$ \citep[see][]{Roginsky:1994}.  
\end{definition}

It is well known (see e.g. \citep{Roginsky:1994,Smith:1955} cf. also
Remark~6.1 in \citep{DepperschmidtPfaffelhuber:2010}) that to prove the 
law of large numbers and the central limit theorem for $(S_{M_t})_{t \ge 0}$  
we need to show that the second moments of $\rho_1-\rho_0$ and
$X_{\rho_1}-X_{\rho_0}$ are bounded (this is done in
Propositions~\ref{prop:rho-bounds} and \ref{prop:X-rho-bounds}). Then, to carry
the result over to $(X_t)_{t \ge 0}$ we  have to prove that the remainder
$(A_t)_{t\ge 0}$ is asymptotically negligible (this is done in
Proposition~\ref{prop:AsympA}). 

\begin{proposition} 
  \label{prop:rho-bounds}
  There exists a positive constant $R^*$ such that for all $x \ge 0$  
  \begin{align}
    \label{eq:6} 
    \mathbb E_x[\rho_0^2] \le R^* \; \text{ and } \; \mathbb E_x [(\rho_1-\rho_0)^2] \le R^*. 
  \end{align}
\end{proposition}
\begin{proof} 
  In the case $x=0$ we have $\rho_0=0$. Consider the case $x>0$. Let $H_1$ be 
  the hitting time of $0$ of the Brownian motion with drift $-\mu$ started in
  $1$. This hitting time has exponential moments (we used this fact already in
  the proof of Lemma~\ref{lem:coupling}). Furthermore set $T=\inf\{t \ge 0:
  X_t -R_t \le 1\}$ and note that $T \le E_{1/2}$, where $E_{1/2}$ is
  independent exponential random variable with rate $1/2$, because as long as
  $X_t -R_t > 1$ the rate at which the reflection boundary jumps into the
  interval $[X_t-1,X_t]$ (and therefore $T$ occurs) is $1/2$. But $T$ also
  occurs if $X_t$ hits the interval $[R_t,R_{t}+1]$. It follows that for any
  initial positions $x \ge 0$, $\rho_0$ is bounded stochastically by the sum
  of independent random variables $E_{1/2}$ and $H_1$, both having exponential
  moments and not depending on $x$. Thus, $\E_x[\rho_0^2]$ is bounded by a
  constant not depending on $x$.

  We write $\rho_1-\rho_0 =(\rho_1-\wt\rho_0)+(\wt\rho_0-\rho_0)$ and argue that
  each of the terms in the brackets has bounded second moments. Note that
  $\wt\rho_0-\rho_0$ is the first jump time of the reflection boundary after
  $\rho_0$. Thus, the finiteness of its second moment follows from
  Lemma~\ref{lem:2nd.mom.kill.time} and that the bound there, also does not
  depend on $x$. The finiteness of the second moment of $(\rho_1- \wt\rho_0)$
  follows by the same argument as the finiteness of the second moment of
  $\rho_0$.
\end{proof}
\begin{proposition} 
  \label{prop:X-rho-bounds}
  There exists a positive constant $R^{**} < \infty$ so that for any $x \ge 0$   
  \begin{align}
    \label{eq:66} 
    \E_x[X_{\rho_0}^2] \le R^{**} \; \text{ and } \; \E_x
    [(X_{\rho_1}-X_{\rho_0})^2]  \le R^{**}.   
  \end{align}
\end{proposition}
\begin{proof}
  Recall that before touching the reflection boundary the process $X_t$ behaves
  as a Brownian motion with drift $-\mu$ and is therefore bounded below by $0$ and
  above by a Brownian motion without drift. Applying the second Wald identity
  and Proposition~\ref{prop:rho-bounds} (with $R^*$ from that proposition) we
  get
  \begin{align} 
    \label{eq:26} 
    \E_x[X_{\rho_0}^2] \le \E_x[B_{\rho_0}^2] = \E_x[\rho_0] \le
    \sqrt{\E_x[\rho_0^2]} \le  \sqrt{R^*}.  
  \end{align} 
  Now we write $X_{\rho_1}-X_{\rho_0}=(X_{\rho_1}- X_{\wt \rho_0})+ (X_{\wt
    \rho_0}-X_{\rho_0})$ and note that as in \eqref{eq:26} the second moment of
  the first term is bounded by $\sqrt{R^*}$. The second moment of the second term 
  is finite according to Lemma~\ref{lem:exp.m.killing} with a bound independent
  of $x$. Taking $R^{**}$ to be the larger of these two bounds \eqref{eq:66}
  follows.
\end{proof}

\begin{proposition}[Asymptotics of $A_t$ and $X^\mu_t-R_t$]  \label{prop:AsympA} 
  We have \[\frac{A_t}{\sqrt t} \to 0 \quad\text{ and } \quad
  \frac{X_t^\mu-R_t}{\sqrt{t}}\to 0 \quad \text{a.s.\ as $t \to \infty$}.\]  
\end{proposition}
We omit the proof here since it is almost the same as the proof of
Proposition~6.7 in \citep{DepperschmidtPfaffelhuber:2010} where the
corresponding result was shown for the Brownian ratchet without drift. (Note
that the definition of $Y_n$ there should be $Y_n=\sup_{t \in
  [\rho_{n-1},\rho_n]} \abs{X_t - X_{\rho_{n-1}}}$. Also the denominator in the
last two displays of that proof should be $\sqrt t$ instead of $t$.)

\subsection{Proof of Theorem~\ref{theoremRBM}}
\label{sec:proof12}
Here we only sketch the proof and refer for details to Section~7 in
\citep{DepperschmidtPfaffelhuber:2010}.     

In the case $\gamma=0$, the law of large numbers in Theorem~\ref{theoremRBM}
holds since $(X_t)_{t \ge 0}$ is then a reflecting Brownian motion with negative
drift $-\mu$ bounded stochastically by a reflecting Brownian motion without drift.
Therefore $X_t/t \to 0$ a.s.\ as $t \to \infty$.

Hence, we assume $\gamma>0$ in the rest of the proof. We use the regeneration
structure from Definition~\ref{def:cumu1} and set  
\begin{align} \label{eq:188}
  \begin{split}
    r & \coloneqq\E_x[\rho_1-\rho_0], \\  m & \coloneqq\E_x[X_{\rho_1}-X_{\rho_0}], \\  
  \beta^2 & \coloneqq
  \text{Var}_x\Big[X_{\rho_1}-X_{\rho_0}-\frac{(\rho_1-\rho_0)m}r\Big].  
  \end{split}
\end{align}
Here, $r, m$ and $\beta^2$ are independent of $x$ due to the
regeneration structure. According to Propositions~\ref{prop:rho-bounds} and
\ref{prop:X-rho-bounds} the temporal and spatial increments 
$\rho_1-\rho_0$ and $X_{\rho_1}-X_{\rho_0}$ have finite second moments. From
that and Proposition~\ref{prop:AsympA} it follows 
\begin{align*}
  \frac{X_t}{t} \xrightarrow{t \to \infty} \frac m r \text{ a.s.}  
\end{align*}

Furthermore, using the CLT for cumulative processes 
\citep[see e.g.][]{Smith:1955,Roginsky:1994} and Proposition~\ref{prop:AsympA}
we obtain that for all $x \in \R$    
\begin{align*}
  \lim_{t \to \infty} \Pr\biggl(\frac{X_t - t m/r}{\beta (t/r)^{1/2}} \le
  x\biggr) & = \lim_{t \to \infty} \Pr\biggl(\frac{A_t}{\beta (t/r)^{1/2}} + \frac{S_{M_t} - 
    t m/r}{\beta (t/r)^{1/2}} \le x\biggr) \\ & = \lim_{t \to
    \infty} \Pr\biggl(\frac{S_{M_t} - t m/r}{\beta (t/r)^{1/2}} \le 
  x\biggr) = \Phi(x),
\end{align*}
where $\Phi$ denotes the distribution function of the standard normal
distribution, and $\beta^2$ and $r$ are as defined in \eqref{eq:188}.
Hence, the central limit theorem holds for $\sigma = \beta/\sqrt{r}$. 

It remains to compute $m/r$. To this end we use the ratio limit theorem for 
Harris recurrent Markov chains  \citep[see e.g.][]{Revuz:1984}. Let $\nu$ denote
the invariant distribution for $(\mathcal Y, \mathcal W, \eta)$. Using the
ratio limit theorem we obtain that  
\begin{align*}
  \frac{m}{r} = \lim_{t\to\infty}\frac{X_t}t= \lim_{t\to\infty} 
  \frac{R_t}t=\lim_{n\to\infty}\frac{R_{\tau_n}}{\tau_n} = \lim_{n \to
    \infty} \frac{\sum_{k=1}^n W_k}{\sum_{k=1}^n \eta_k} =
  \frac{\E_\nu[W_1]}{\E_\nu[\eta_1]}, 
\end{align*} 
where for the second equality we have used Proposition~\ref{prop:AsympA}. 
We recall that $\E_\nu[W_1]=\E_\nu[Y_1]$.  

Let $v:[0,\infty)^2 \to [0,\infty)$, $(\mu,\gamma) \mapsto v(\mu,\gamma)$, 
denote the speed $m/r$ as a function of $\mu$ and $\gamma$.  In the case
$\gamma=\frac12$ we obtain from Proposition~\ref{prop:exp-under-pi} 
\begin{align*}
  v\left(\mu,\frac12\right) = \frac{\E_\nu[Y_1]}{ \E_\nu[\tau_1]} = -\frac{\mu  
    Ai(\mu^2)+Ai'(\mu^2)}{2Ai(\mu^2)} =
  -\frac12\Bigl(\frac{Ai'(\mu^2)}{Ai(\mu^2)}+\mu\Bigr).  
\end{align*}

Now let $\mu\ge 0$ and $\gamma >0 $ be given. Using the scaling property (see
Lemma~\ref{lem:scaling}) we obtain  
\begin{align*}
  (X^{\gamma,\mu}_t,R_t^{\gamma,\mu}) \stackrel{d}{=}  
  (2\gamma)^{-1/3} \left(X^{1/2,(2\gamma)^{-1/3}\mu}_{(2\gamma)^{2/3}t},
  R_{(2\gamma)^{2/3}t}^{1/2,(2\gamma)^{-1/3}\mu}\right).
\end{align*}
Thus, we have 
\begin{align*}
  v(\mu,\gamma) & = \lim_{t\to\infty} \frac{\E\left[X^{\gamma,\mu}_t\right]}{t}
   =  (2\gamma)^{1/3}\lim_{t\to\infty}
  \frac{\E\left[X^{1/2,(2\gamma)^{-1/3}\mu}_{(2\gamma)^{2/3}t}\right]}{(2\gamma)^{2/3}t}  
  \\ & = (2\gamma)^{1/3} v\Bigl((2\gamma)^{-1/3}\mu,\frac12\Bigr) \\ 
 & = -\frac{(2\gamma)^{1/3}}{2} 
  \Bigl(\frac{Ai'((2\gamma)^{-2/3}\mu^2)}{Ai((2\gamma)^{-2/3}\mu^2)} 
  +(2\gamma)^{-1/3}\mu\Bigr)  
\end{align*}
which concludes the proof of Theorem~\ref{theoremRBM}.

\section{Ornstein-Uhlenbeck ratchet}
\label{sec:OU-ratchet}

In this section we carry out the same program for the Ornstein-Uhlenbeck ratchet
as we did for the Brownian ratchet. The arguments in many of the proofs here are 
similar to the corresponding proofs in the previous section. Therefore some
proofs in this section will be sketchy.

\subsection{Graphical construction}
\label{sec:refl-OU}

Let us first recall the definition of the reflecting Ornstein-Uhlenbeck (OU) 
process with infinitesimal drift $-\mu x$ and unit variance.  
Suppose that $B=\left(B(t)\right)_{t\ge 0}$ is a standard Brownian motion 
starting in $0$ and let $x_0\ge 0$. Then using the representation of the OU
process as a time changed Brownian motion, we obtain that  $\wh Z\coloneqq
\left(\wh  Z_t\right)_{t \ge 0}$, defined by    
\begin{align}
  \label{eq:OUfromBM}
  \wh Z_t \coloneqq  \Abs{x_0 e^{-\mu t} + \frac1{\sqrt{2\mu}} e^{-\mu t} B(e^{2\mu t}  
    -1)}
\end{align}
is a reflecting  OU process with infinitesimal drift $-\mu x$ and unit
variance starting in $x_0$. It is a diffusion process on $[0,\infty)$ associated 
with infinitesimal operator $\wh A^\mu$ acting on   
\begin{align*}
  \mathcal D(\wh A^\mu)\coloneqq  \{f \in C_b^2(\R_+) : f'(0+) =0\}  
\end{align*}
as follows: 
\begin{align}
  \label{eq:gener-OU}
  \wh A^\mu f(x) \coloneqq \frac 12 f''(x) -\mu x f'(x). 
\end{align}

The graphical construction in the following definition (see
Figure~\ref{fig:ou-ratchet}) is different from the graphical construction of the 
Brownian ratchet. Here we use a family of independent Brownian motions to
construct reflected OU processes between the jumps of the ratchet. At any jump
time a new reflected OU process starts in an initial value chosen uniformly
between the state of the previous process and zero. Then we stick this
``peaces'' together to obtain the OU ratchet.

\begin{figure}[htb]
  \begin{center}
    \includegraphics[width=0.98\textwidth]{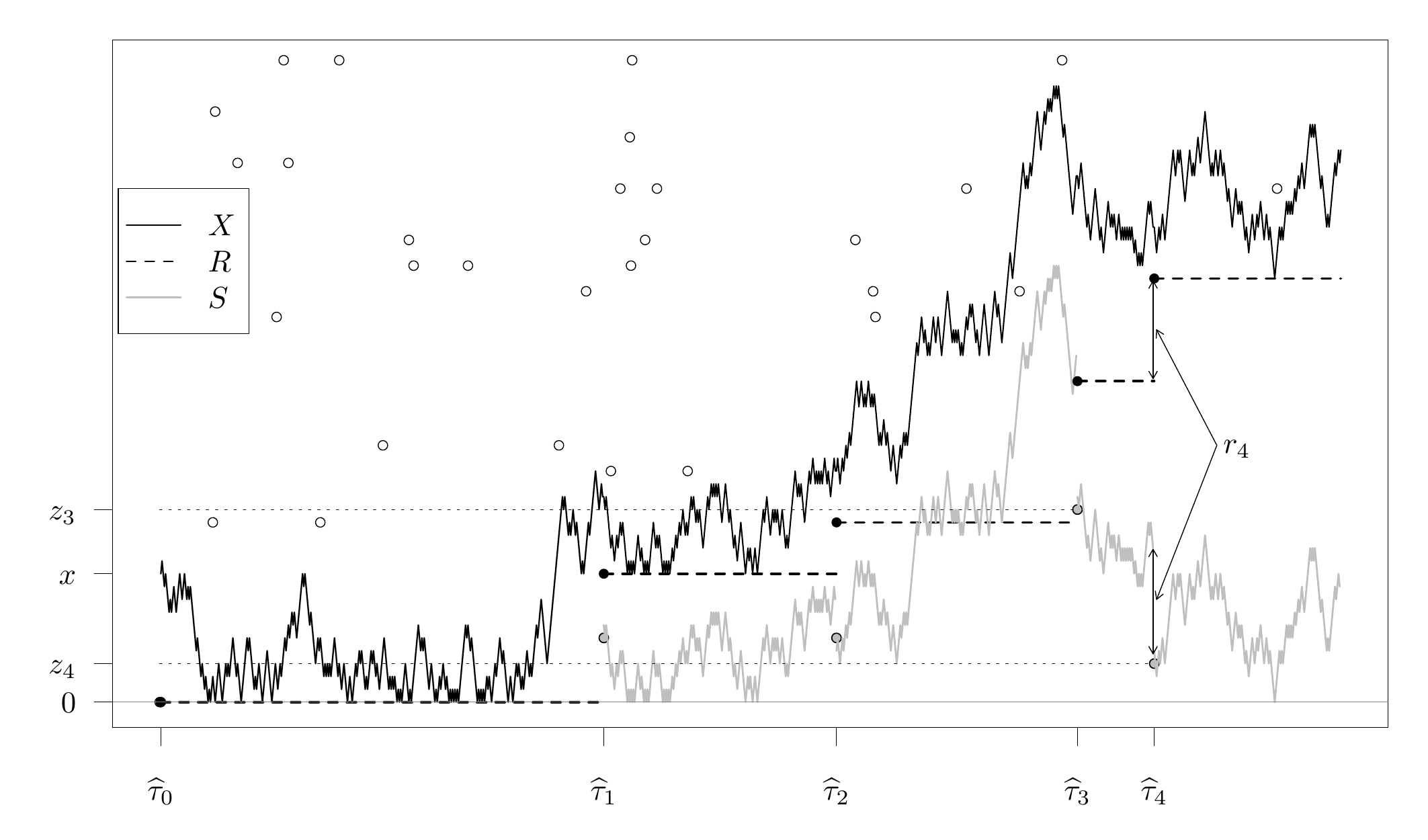}%  
    \caption{Graphical construction of the Ornstein-Uhlenbeck ratchet.} 
    \label{fig:ou-ratchet}
  \end{center}
\end{figure}

\begin{definition}[Graphical construction of the Ornstein-Uhlenbeck ratchet] 
  \label{def:gr-constr-OUR} 
  Let $N^{\gamma}$ be a Poisson process and let $\left((B^{(i)}(t))_{t\ge  
      0}\right)_{i=0,1,\dots}$ be independent standard Brownian motions.  
  We define a sequence of stopping times $(\wh\tau_n)_{n =0,1,\dots}$ and a 
  sequence of  Ornstein-Uhlenbeck processes $(\wh S^{(n)})_{n=0,1,\dots}$
  with $\wh S^{(n)}=(S^{(n)}_t)_{t \ge \tau_n}$  reflecting at $0$ as follows:   
  \begin{align} 
    \label{eq:23} 
    \wh\tau_0 & = 0, \\
    \label{eq:10} 
    \wh S^{(0)}_t & = \Abs{x_0e^{-\mu t} + e^{-\mu t}\frac1{\sqrt{2\mu}} 
      B^{(0)}(e^{2\mu  t}-1)}.    
  \end{align}
  Given $\wh\tau_{n-1}$ and $\wh S^{(n-1)}$ for some  $n \ge 1$ we set  
  \begin{align}
    \label{eq:OU-tau} 
    \wh\tau_{n} & =\inf\{t > \wh\tau_{n-1}: N^\gamma \cap [0,\wh S^{(n-1)} (t)]\times
    \{t\} \ne   \emptyset\bigr\}   
  \end{align}
and let $z_{n}$ be the space component of the almost surely
unique element of $N^\gamma \cap [0,\wh S^{(n-1)}_{\wh\tau_{n}}]\times
\{\wh\tau_{n}\}$. Furthermore we set $r_{n}=\wh S^{(n-1)}_{\wh \tau_{n}}-z_{n}$
and for $t \ge \wh \tau_{n}$ we define  
\begin{align}
  \label{eq:25} 
  \wh S^{(n)}_t = \Abs{z_n e^{-\mu(t-\wh\tau_n)} + e^{-\mu(t-\wh \tau_n)}
    \frac{1}{\sqrt{2\mu}} B^{(n)}(e^{2\mu(t-\wh\tau_n)}-1)}. 
\end{align}
Finally we set for $t \in [\wh\tau_{n},\wh\tau_{n+1})$   
  \begin{align} \label{eq:24} 
    \wh R_t & = \sum_{i \le n} r_i  \quad \text{ and } \quad   \wh S_t =
    \wh S^{(n)}_t \\ \intertext{and for $t\ge 0$}
    \wh X_t & = \wh R_t + \wh S_t.     
  \end{align}
  Note that by construction we have 
  \begin{align}
    \label{eq:28} 
    \wh S \le \wh S^{(0)} 
  \end{align}
  stochastically.   \qed 
\end{definition}

The following lemma is the analogue of Lemma~\ref{lem:gr-constr-BR}. Though the
graphical construction there is somewhat different the proof is similar and will
be omitted here. 
\begin{lemma} \label{lem:gr-constr-OUR} 
  The process $(\wh\XX,\wh\RR) \coloneqq  (\widehat X_t, \widehat R_t)_{t \ge 0}$ is 
  a $(\gamma,\mu)$-Ornstein-Uhlenbeck rat\-chet starting in $(x_0,0)$.    
\end{lemma}

Now we  construct a coupling of two Ornstein-Uhlenbeck ratchets starting in
$((x_0,0),( x'_0,0))$, where we assume $x_0 \geq  x'_0 \geq 0$ without loss 
of generality. Let the Poisson process $N^\gamma$ and a sequence
$((B^{(i)}_t)_{t\geq0})_{i = 0,1,\dots}$ of independent standard Brownian
motions be given as before. To construct the coupling  
\begin{align*}
  \bigl((\wh X_t,\wh R_t),( \wh{X}'_t,\wh{R}'_t)\bigr)_{t\ge 0}, 
\end{align*}
set $\wh \tau_0=\wh {\tau'}_0=0$, $\wh {S}_0^{(0)}=x$, $\wh {S'}_0^{(0)}=   
x'$ and define the sequences $(\wh \tau_n)_{n\ge 0}$ and $(\wh \tau'_n)_{n
  \ge 0}$ and the processes $(\wh S^{(n)})_{n\ge 0}$ and $( \wh {S'}^{(n)})_{n\ge 0}$
as in Definition~\ref{def:gr-constr-OUR}. Furthermore define     
$\wh S$, $\wh {S}'$, $(\wh X_t,\wh R_t)_{t\ge 0}$ and $(\wh X'_t,\wh R'_t)_{t\ge
  0}$ as in \eqref{eq:24}. Then define the coupling time by   
\begin{align*} 
  \wh T_{\textnormal{coupl}}\coloneqq \inf\{t>0: \wh S_t = \wh S'_t\}.   
\end{align*}
Since we can use the same Brownian motions and the same Poisson  
process for both ratchets we have $\wh S_t = \wh S'_t$ for all  $t \ge
\wh T_{\textnormal{coupl}}$. Thus, on the event $\{\wh T_{\textnormal{coupl}} < \infty\}$,
there are $\wh n, \wh n'$ such that for $k \ge 0$ 
\begin{align*} 
    \wh \tau_{n+k}  =\wh \tau'_{{n} +k}, &&
    \wh X_{\wh \tau_{n+k+1}}-\wh X_{\wh \tau_{n+k}}  =\wh X'_{\wh \tau_{n+k+1}}
    -\wh X'_{\wh \tau_{n+k}},   &&
    \wh R_{\wh \tau_{n+k+1}}-\wh R_{\wh \tau_{n+k}}  =\wh R'_{\wh \tau_{n+k+1}}
    -\wh R'_{\wh \tau_{n+k}}. 
\end{align*}
The following lemma shows that the coupling is successful with probability one.  
\begin{lemma}[Exponential moments of the coupling time] \label{lem:exp-mom-coupl-OUR}
  For any $\mu >0$ there is $\alpha>0$ so that
  \begin{align*}
    \mathbb E [ e^{\alpha \wh T_{\textnormal{coupl}}}] < \infty.    
 \end{align*}
\end{lemma}
\begin{proof}
  We only need to consider the case $x_0 > x'_0$. In that case $\wh S_t \ge \wh
  S_t'$ for all $t$.  Furthermore, from \eqref{eq:28} it follows that $\wh S$ is
  stochastically dominated by the reflected Ornstein-Uhlenbeck process $\wh
  S^{(0)}$.  Thus, $\wh T_{\tn{coupl}}$ is stochastically bounded by the hitting
  time of $0$, say $H_0$, by the process $\wh S^{(0)}$. For $H_0$ we have
  \citep[see][p.~542]{BorodinSalminen:2002}
  \begin{align*}
    \Pr_{x_0}[H_0 >t] = \mathrm{Erf}(x_0/\sqrt{2(e^{2\mu t} -1)}), 
  \end{align*}
  where $\mathrm{Erf}(x) = 2\pi^{-1/2} \int_0^x e^{-u^2}\, du =
  2\pi^{-1/2}\sum_{n=0}^\infty \frac{(-1)^n x^{2n+1}}{n! (2n+1)}$. Thus,  as $t 
  \to \infty$ we obtain
  \begin{align*}
        \Pr_{x_0}[H_0 >t] = \frac2{\sqrt\pi} \frac{x_0}{\sqrt{2(e^{2\mu t} -1)}}
        + o (e^{-3\mu t}) \le x_0 C e^{-\mu t} 
  \end{align*}
for suitably chosen positive $C$. 

\end{proof}

\subsection{Green function of the killed reflected Ornstein-Uhlenbeck process}
\label{sec:Greem-OU}
In this subsection we carry out analogous computations to those in 
Subsection~\ref{sec:green}; recall in particular Remark~\ref{rem:Green-gener}. 

\begin{definition}[Killed reflecting Ornstein-Uhlenbeck process]
  \label{rem:killed-rou} 
  Let $\wh Z \coloneqq (\wh Z_t)_{t\ge 0}$ denote reflecting Ornstein-Uhlenbeck
  process with infinitesimal drift $-\mu x$ and unit infinitesimal variance
  starting in $x \ge 0$ and let $\Pr_{x}$  denote the corresponding law on
  the paths space and $\E_{x}$ the expectation under this law. Furthermore 
  using an exponentially distributed rate $1$ random variable $\xi$
  independent of $\wh Z$ we define the \emph{killing time} by  
  \begin{align*}
    \wh\tau=\inf\{t >0: \gamma \int_0^t \wh Z_s\,ds \ge \xi\}.     
  \end{align*}
  The \emph{reflecting Ornstein-Uhlenbeck process killed at rate $\gamma \wh
    Z$} is defined as the process $\wh Z^K \coloneqq (\wh Z^K_t)_{t\ge 0}$
  with $\wh Z^K_t = \wh Z_t$ for $t \in [0,\tau)$ and $\wh Z^K_t = \Delta$ 
  for $t \ge \tau$ where $\Delta \not\in \R$ is the \emph{cemetery state.}    
  The infinitesimal operator $\wh A^{\mu,\gamma}$ corresponding to $\wh Z^K$
  acts on $C^2$ functions $f:[0,\infty) \to \R$ satisfying $f'(0+)=0$ as
  follows  
  \begin{align} 
    \label{eq:gen-killed-OU}  
    \wh A^{\mu,\gamma} f(x) = \frac12 f''(x)  - \mu xf'(x) -\gamma x f(x).    
  \end{align}
  The speed measure and the killing measure corresponding to the killed
  Ornstein-Uhlenbeck process are given by %\citep[see][p.~17]{BorodinSalminen:2002} 
  \begin{align}
    \label{eq:speed-killing-measure} 
    m(dx)=m(x) \, dx \coloneqq 2e^{-\mu x^2}\,dx \quad \text{and} \quad  k(dx)=k(x)
    \,dx \coloneqq  2\gamma xe^{- \mu x^2}\, dx.
  \end{align}
\end{definition}
In the case of Ornstein-Uhlenbeck ratchet the \emph{confluent hypergeometric
  functions} play a similar role as the Airy functions in the case of Brownian
ratchet. In the following remark we collect some of their properties that will
be needed in the sequel. We refer to \citep[][Ch.~13]{NIST-Handbook} for most of
the properties and for more information on confluent hypergeometric functions.
\begin{remark}[Confluent hypergeometric functions]  \label{rem:confl-funct}
 Assume that $b \not \in\Z$
  and consider the Kummer equation (also known as the confluent hypergeometric 
  equation)  
  \begin{align} 
    \label{eq:11} 
    x z''(x) + (b-x)z'(x)-a z(x)=0.   
  \end{align}
  Note that solutions also exist in the case $b \in \Z$ but some of relations 
  that we are going to recall in this remark and use in the following may not
  hold in this case.  
  \begin{asparaenum} 
  \item \emph{Standard solutions:} One standard solution of that equation is
    given by the function $x\mapsto M(a,b,x)$ which is defined by   
    \begin{align}
      \label{eq:M} 
      M(a,b,x) & \coloneqq  \sum_{n=0}^\infty \frac{(a)_n}{(b)_n}\frac{x^n}{n!},   
    \end{align}
    where $(a)_n$ is the Pochhammer's symbol defined by $(a)_0=1$ and
    $(a)_n=a(a+1) \cdots (a+n-1)$ for $n \ge 1$. Another standard solution is
    given by the function $x\mapsto U (a,b,x)$ which can be defined (in the case
    $b\not\in \Z$) by
    \begin{align}
      \label{eq:U} 
      U(a,b,x) & \coloneqq  \frac{\Gamma(1-b)}{\Gamma(a-b+1)} M(a,b,x) +
      \frac{\Gamma(b-1)}{\Gamma(a)} x^{1-b} M(1+a-b,2-b,x),   
    \end{align}
    where $\Gamma$ denotes the Gamma function. If $a\ne 0,-1,-2,\dots$, then $U$
    and $M$ are independent solutions of \eqref{eq:11}. For $a = 0, -1, \dots$
    we have $\abs{\Gamma(a)}=\infty$ and the second summand on the right hand
    side of \eqref{eq:U} vanishes. Thus, as is easily seen from \eqref{eq:M} and
    \eqref{eq:U}, in that case both $U$ and $M$ are polynomials which are equal
    up to a multiplicative constant. The system of independent solutions in the
    case $a = 0, -1, \dots$ is given by
    \begin{align}
      \label{eq:ind-sol-MM}
      U(a,b,z) \; \text{ and }  \;  z^{1-b} M(a-b+1,2-b,z).  
    \end{align}
  \item \emph{Behaviour in the neighbourhood of zero:} For any $a,b \in \R$ 
    \begin{align}\label{eq:39} 
      M(a,b,x) = 1 + \mathcal{O}(x). 
    \end{align}
    For $b \in (0,1)$    
    \begin{align}
      \label{eq:35}
      U(a,b,x) = \frac{\Gamma(1-b)}{\Gamma(a-b+1)} + \mathcal{O} (x^{1-b}). 
    \end{align}
    For $b \in (1,2)$ 
    \begin{align}
      \label{eq:34}
      U(a,b,x) = \frac{\Gamma(b-1)}{\Gamma(a)} x^{1-b} +
      \frac{\Gamma(1-b)}{\Gamma(a-b+1)} + \mathcal{O} (x^{2-b}).  
    \end{align}
    % In particular in this case we have $U(a,b,x) \uparrow \infty$ as $x
    % \downarrow 0$. 
  \item \emph{Behaviour at infinity:}  For $a \ne -1,-2,\dots$  
    and $b >0$ 
    \begin{align}
      \label{eq:M-asympt}
      M(a,b,x) \sim \frac{x^{a-b}}{\Gamma(a)}e^{x} \text{ as } x \to \infty.
    \end{align}
    Here, as usual, we write $g(x) \sim h(x)$ if $g(x)/h(x) \to 1$ as $x \to
    \infty$. Furthermore
    \begin{align}
      \label{eq:U-asympt}  
      U(a,b,x) \sim x^{-a} \text{ as } x \to \infty. 
    \end{align}
    Note also that $U$ is uniquely determined by this property.  
  \item \emph{Bounds for positive zeros of $U$ and $M$:} If $a, b \ge 0$ then
    $M$ has no zeros on $[0,\infty)$. Let $P(a,b)$ be the number of positive
    zeros of $U(a,b,x)$. First we note that using the Kummer transformation
    \begin{align}
      \label{eq:Kummer}
     U(a,b,x)=x^{1-b} U(a-b+1,2-b,x) 
    \end{align}
    we get
    \begin{align}
      \label{eq:55}
      P(a,b) = P(a-b+1,2-b).
    \end{align}
    If $a$, $b$ and $a+b-1$ are non-integers, $b < 1$ and $a+1 \ge b$ then $P(a,b) = 0$. 
    Let $a < 0$ and $1 \le b \le 2$. If $x_0$ is a positive zero of $U(a,b,x)$  
     then  by (2.19) in \citep{MR1052433}   
     \begin{align}
       \label{eq:48}
       x_0 < 4(\frac b 2 -a). 
     \end{align} 
     
  \item \emph{Differentiation formulas:} We have  
    \begin{align}
      \label{eq:diff-U}
      \frac{d}{dx} U(a,b,x)& =-a U (a+1,b+1,x) \\ \intertext{and}   
      \label{eq:diff-M} 
      \frac{d}{dx} M(a,b,x)  &=  \frac{a}{b} M (a+1,b+1,x).  
    \end{align}
  \item \emph{Recurrence relation:} There are many recurrence relations for $U$
    and $M$. We will need the following (it follows from (13.4.25) in
    \citep{AbramowitzStegun:1972} and  \eqref{eq:diff-U})   
    \begin{align}
      \label{eq:51}
      U(a,b+1,x) = U(a,b,x) + a U(a+1,b+1,x). 
    \end{align}
    
  \hfill \qed 
  \end{asparaenum}
\end{remark}
By straightforward computation one can show that if a function
$f_{\gamma,\mu}(x)$ is a solution of the Kummer equation \eqref{eq:11} with 
$a=-\frac{\gamma^2}{4\mu^3}$ and $b=\frac12$, then
\begin{align} \label{eq:52} 
  F_{\gamma,\mu}(x)=e^{-\gamma
    x/\mu}f_{\gamma,\mu}\left(\left(\frac\gamma{\mu^{3/2}}+\sqrt \mu x\right)^2\right) 
\end{align}
is a solution of
\begin{align}
  \label{eq:exp-confl}
  \wh A^{\mu,\gamma} f = 0.  
\end{align}
As in the case of killed Brownian motion with negative drift there are positive
solutions $\wh \phi$ and $\wh \psi$ of \eqref{eq:exp-confl} on $[0,\infty)$ with
Wronskian $w(\wh\psi,\wh\phi) \coloneqq \wh\psi'(x)\wh\phi(x) -
\wh\psi(x)\wh\phi'(x)=1$ such that
\begin{align} 
  \label{eq:whphiwhpsi} 
  \begin{cases}
   \text{$\wh \phi$ is decreasing and $\wh \phi(x)\to 0$ as $x \to
     \infty$,} & \\
    \text{$\wh \psi$ is increasing and $\wh \psi'(0)=0$}.  & 
  \end{cases}
\end{align}
Then the Green function of the killed Ornstein-Uhlenbeck process is given by  
\begin{align}
  \label{D:OPUgreenfunction}
  G(x,y)\coloneqq  
  \begin{cases}
    \wh \phi(x)\wh\psi(y) & : \; 0 \le y \le x,\\ 
    \wh \psi(x) \wh\phi(y) & : \; 0 \le x \le y.   
  \end{cases}
\end{align}

We define the function $x \mapsto p(x)$ by  
\begin{align}
  \label{eq:40} 
  p(x) = \frac\gamma{\mu^{3/2}}+\sqrt \mu x.  
\end{align}
For general $\gamma$ and $\mu$ the solutions of \eqref{eq:11} with
$a=-\frac{\gamma^2}{4\mu^3}$ and $b=\frac12$ with $a=-\frac{\gamma^2}{4\mu^3}$
are hard to deal with. Thus, we will not formulate an analogue of 
Lemma~\ref{lem:green-krbm} for the killed reflecting Ornstein-Uhlenbeck process.
In the following lemma we will identify the decreasing solution $\wh\phi$. For
$\wh\psi$ we will assume that for any given $\mu,\gamma >0$ there is a suitable 
linear combination, say  $\wh M(x)$, of   
\begin{align} 
  \label{eq:30} 
  M\left(-\frac{\gamma^2}{4\mu^3},\frac12,p^2(x) \right) \; & \text{ and } \;
  U\left(-\frac{\gamma^2}{4\mu^3},\frac12,p^2(x) \right) \\ \intertext{ if
    $\frac{\gamma^2}{4\mu^3} \ne 1, 2,\dots$, and} \label{eq:32} p(x)
  M\left(-\frac{\gamma^2}{4\mu^3}+\frac12,\frac32,p^2(x) \right) \; & \text{ and
  } \; U\left(-\frac{\gamma^2}{4\mu^3},\frac12,p^2(x) \right)
\end{align}
if $\frac{\gamma^2}{4\mu^3} = 1, 2,\dots$,  so that the function 
$\wh\psi:[0,\infty) \to [0,\infty)$ defined by        
\begin{align} \label{eq:psiphi}
    \wh\psi(x) =e^{-\gamma x /\mu} \wh M(x)  
\end{align}
satisfies the condition in \eqref{eq:whphiwhpsi}. 

\begin{lemma} 
  \label{lem:whphi-decr} 
  The function $x \mapsto \wh \phi(x)$ defined on $[0,\infty)$ by 
  \begin{align} \label{eq:wh.phi}
    \wh\phi(x) =e^{-\gamma x /\mu}
    U\left(-\frac{\gamma^2}{4\mu^3},\frac12,p^2(x) \right)   
  \end{align}
  is a positive decreasing solution of \eqref{eq:exp-confl} with $\wh\phi(x) \to
  0$ as $x \to \infty$.   
\end{lemma}
\begin{proof}
  Since $\wh \phi$ is of the form \eqref{eq:52} it is a solution of \eqref{eq:exp-confl}.  
  By \eqref{eq:U-asympt} it is clear that $\wh\phi(x)>0$ for large $x$ and 
  $\wh\phi(x)\to 0$ as $x \to \infty$. If we show that $\wh \phi''(x)>0$ for $x
  >0$, i.e.\ that $\wh\phi$ is convex, then it follows that $\wh\phi$ is
  positive and decreasing on $(0,\infty)$.  
  We have
  \begin{align*}
    \wh\phi''(x)  & =2\mu x \wh\phi'(x)+2\gamma x \wh\phi(x) \\ 
    & = 2 \mu x e^{-\frac{\gamma x}\mu} 2 \sqrt{\mu}
    p(x) U'(-\frac{\gamma^2}{4\mu^3},\frac12,p^2(x))  
    \\ & = 4 \mu^{3/2} x e^{-\frac{\gamma x}\mu} p(x) \frac{\gamma^2}{4\mu^3}
    U(1-\frac{\gamma^2}{4\mu^3},\frac32,p^2(x)),  
  \end{align*}
  where the last step follows by \eqref{eq:diff-U}. Again by \eqref{eq:U-asympt}
  it is clear that $\wh\phi''(x)>0$ for large $x$. Thus, to show that
  $\wh\phi''(x) > 0$ for $x >0$ it is enough to show that this function has no
  positive zeros.

  Using \eqref{eq:U} and the definition of $M$ one can easily compute that
  $U(\frac12,\frac32,p^2(x)) = \frac1{p(x)}$ and $U(0,\frac32,p^2(x))=1$. Thus
  in the case $\frac{\gamma^2}{4\mu^3}\in\{\frac12,1\}$  we have $\wh\phi''(x)
  >0$ for $x \ge 0$.  

  Let us now consider the case $\frac{\gamma^2}{4\mu^3} \in (0,1)
  \setminus\{\frac12\}$. Then by \eqref{eq:55} the number of positive zeros of  $x \mapsto
  U(1-\frac{\gamma^2}{4\mu^3},\frac32,x)$ equals the number of positive zeros of $x \mapsto
  U(\frac12-\frac{\gamma^2}{4\mu^3},\frac12,x)$. Since 
  \begin{align*}
    \frac12-\frac{\gamma^2}{4\mu^3}+1 \ge \frac12,   
  \end{align*}
  by Remark~\ref{rem:confl-funct}(iv), the function $x \mapsto 
  U(1-\frac{\gamma^2}{4\mu^3},\frac32,x)$ has no positive zeros, which implies 
  that $\wh\phi''(x) >0$ for $x >0$.  

  If $\frac{\gamma^2}{4\mu^3} > 1$ then, by \eqref{eq:48} all positive zeros of
  $x \mapsto U(1-\frac{\gamma^2}{4\mu^3},\frac32,x)$ are bounded by $
  \gamma^2/\mu^3-1$. Since the set $\{x \ge 0: p^2(x) < \gamma^2/\mu^3 -1 \}$
  is empty, the function $\wh\phi''$ is positive on $(0, \infty)$.
\end{proof}

In the following example we compute the function $\wh\psi$ in a special case. 
\begin{example}[$\wh \psi(x)$ in a special case] 
  Assume $0< \frac{\gamma^2}{4\mu^3} < \frac12$, let $\wh \phi$ be as in
  \eqref{eq:wh.phi} and define
  \begin{align*}
    \wh \psi_1 (x)  \coloneqq - e^{-\gamma x/\mu}
    M\left(-\frac{\gamma^2}{4\mu^3},\frac12,p^2(x)\right)    
  \end{align*}
  and 
  \begin{align*}
    \wh \psi (x) \coloneqq \wh\psi_1(x) + C\wh\phi(x), \; \text{   with } \;
    C=-\frac{\wh\psi_1'(0)}{\wh\phi'(0)}. 
  \end{align*}
  By the choice of $C$ we have $\wh\psi'(0)=0$ as required.  Furthermore,
  Remark~\ref{rem:confl-funct}(iii) implies that $\wh\psi(x) \to \infty$ for
  $x\to\infty$. To show that $\wh\psi$ satisfies the conditions from
  \eqref{eq:whphiwhpsi} we need to show convexity of $\wh\psi$, i.e.\ positivity of
  $\wh\psi''$ on $[0,\infty)$ and $\wh\psi(0)>0$.  

  Using \eqref{eq:diff-M} we obtain 
  \begin{align} \label{eq:37}
    \begin{split}
      \frac12 \wh\psi''_1 (x) & =\mu x \wh\psi_1'(x)+\gamma x \wh\psi_1(x) = - 2
      \mu^{3/2} x e^{-\frac{\gamma x}\mu}
      p(x)  M'\left(-\frac{\gamma^2}{4\mu^3},\frac12,p^2(x)\right)   \\
      & = \frac{\gamma^2}{\mu^{3/2}} x e^{-\frac{\gamma x}\mu} p(x)
      M\left(1-\frac{\gamma^2}{4\mu^3},\frac32,p^2(x)\right).
    \end{split}
  \end{align}
  By assumption we have $1-\frac{\gamma^2}{4\mu^3} \ge 0$ and therefore 
  Remark~\ref{rem:confl-funct}(iv) implies that the function $x \mapsto 
  M(1-\frac{\gamma^2}{4\mu^3},\frac32,p^2(x))$ has no zeros on $(0,\infty)$.
  Since by  Remark~\ref{rem:confl-funct}(iii) this function is positive for
  large $x$ we obtain $\wh\psi''_1(x)>0$ and $\wh\psi'_1(x) >0$ for $x > 0$.  

  In Lemma~\ref{lem:whphi-decr} we have shown that $\wh\phi''$ is positive and
  $\wh\phi'$ is negative on $(0,\infty)$. Together with positivity of
  $\wh\psi_1''$ and $\wh\psi_1'$ it follows 
  \begin{align*}
    \wh\psi''(x) = \wh\psi_1''(x) - \frac{\wh\psi_1'(0)}{\wh\phi'(0)} \wh\phi''(x)
    >0, \quad x >0.  
  \end{align*}
  Finally we have
  \begin{align*}
    \wh\psi(0) & = \wh\psi_1(0) - \frac{\wh\psi_1'(0)}{\wh\phi'(0)} \wh\phi(0) =
    - \frac{\wh\psi_1'(0)\wh\phi(0) - \wh\psi_1(0)\wh\phi'(0)}{\wh\phi'(0)} = -
    \frac{w(\wh\psi_1,\phi)}{\wh\phi'(0)}.   
  \end{align*}
  Here $w(\wh\psi_1,\wh\phi)= \wh\psi_1'(x)\wh\phi(x) - \wh\psi_1(x)\wh\phi'(x)$
  is a constant independent of $x$, and since $\wh\psi_1'(x)$, $\wh\phi(x)$ and
  $\wh\psi_1(x)\wh\phi'(x)$ are positive for large $x$ and $\wh\phi'(x)$ is
  negative this constant must be positive. Note also that, as is easily
  computed, $w(\wh\psi_1,\wh\phi)=w(\wh\psi,\wh\phi)$. In order this to be one, 
  we would need to normalize $\wh\psi$ by that value. Together with
  $\wh\phi'(0)<0$ it follows $\wh\psi(0)>0$ as required. \qed
\end{example}

\begin{lemma}[Upper bounds of $\wh\phi(x)$ and $\wh\psi(x)$] 
  Let $\wh\psi$ and $\wh\phi$ be as defined in \eqref{eq:psiphi} and
   \eqref{eq:wh.phi}. For any positive $\gamma$ and $\mu$ there are
  finite positive $K_1$ and $K_2$ such that  
  \begin{align}\label{eq:27} 
    \wh\psi(x) \le K_1 e^{\gamma x/\mu + \mu x^2} \quad \text{ and }
    \quad \wh\phi(x) \le K_2 e^{-\gamma x/\mu}.
  \end{align}
\end{lemma}
\begin{proof} 
  Assume first that $\frac{\gamma^2}{4\mu^3} \ne 1, 2,\dots$. Then two
  independent solutions of the equation \eqref{eq:11} are given by
  \eqref{eq:30}. In view of \eqref{eq:M-asympt}, \eqref{eq:U-asympt},
  \eqref{eq:39} and \eqref{eq:35} it is clear that any linear combination of the
  functions from \eqref{eq:30} multiplied by $e^{-\gamma x/\mu}$ is bounded on
  any interval of the form $[0,t]$ for $t < \infty$. Furthermore, as $x \to
  \infty$, it grows at most as a constant times $e^{-\gamma x/\mu}e^{p^2(x)} =
  e^{\gamma^2/\mu^3 + \gamma x/\mu + \mu x^2}$ from which the first part of the
  assertion follows.   If $\frac{\gamma^2}{4\mu^3} = 1, 2,\dots$, then two
  independent solutions of the equation \eqref{eq:11} are given by the functions
  in \eqref{eq:32}. Using similar arguments as above one can show the first part
  of the assertion also in this case.  

  The second part of the assertion follows from the definition of $\wh\phi$ and
  \eqref{eq:U-asympt}. 
\end{proof}

We set 
\begin{align} 
  \label{eq:tildephiOU}
  \wh\Phi(x) & = e^{-\mu x^2}\wh\phi(x),  \\  
  \label{eq:tildepsiOU} 
  \wh\Psi(x) &  = e^{-\mu x^2} \wh\psi(x),  
\end{align}
and note that $\wh\Phi$ and $\wh\Psi$ solve the differential equation
\begin{align}
  \label{eq:tildeeq}
  \frac12 u''(x)=(\gamma x -\mu) u(x) - \mu x u'(x).   
\end{align}
From \eqref{eq:27} it follows 
\begin{align} \label{eq:43} 
  \wh\Psi(x) \le K_1 e^{\gamma x/\mu} \qquad  \text{ and } \qquad\wh\Phi(x) 
  \le K_2 e^{-\gamma x/\mu-\mu x^2}.  
\end{align}
Furthermore, 
\begin{align}
  \label{eq:45}
  \wh\phi(x)\wh \Psi'(x) - \wh \psi(x) \wh \Phi'(x) = \wh\phi(x)\wh \psi'(x) -
  \wh \psi(x) \wh \phi'(x) = 1 
\end{align}
and
\begin{align}
  \label{eq:47}
  \wh\Psi'(0) =0. 
\end{align}

\begin{remark}[Density of the killing position and expected killing
  time] \label{rem:dens-kill-pos-OU} \leavevmode% 
  \\ 
  As in Remark~\ref{rem:exp-kill} in the Brownian case, the density of the
  position at killing time of the killed Ornstein-Uhlenbeck process starting in
  $x$  is given by      
  \begin{align}
    \label{equ:density}
    f_x(y)  = G(x,y) k(y) = G(x,y) 2\gamma ye^{- \mu y^2}.   
  \end{align}
  The expected killing time is given by 
  \begin{align}
    \label{eq:38}
    \begin{split}
      \E_x[\wh\tau]& = \int_0^\infty G(x,y) m(y)\, dy = \int_0^\infty G(x,y)
      2e^{-\mu y^2} \, dy \\ & = 2\left(\wh \phi(x) \int_0^x \wh\Psi(y)\,dy +
        \wh\psi(x) \int_x^\infty \wh\Phi(y)\, dy \right).
    \end{split}
  \end{align}
\end{remark}

\begin{lemma}[Exponential moments of the killing position]  
  \label{lem:exp.m.killing.OUP} \leavevmode \\    
  For $\alpha < \gamma/\mu$ and  any $x \ge 0$  we have
  \begin{align}
    \label{eq:exp.m.killing.OUP}
    \mathbb E_x [e^{\alpha X_{\tau-}}] < \infty.  
  \end{align}  
\end{lemma}
\begin{proof}
  We have 
  \begin{align*}
    \E_x [e^{\alpha X_{\tau-}}] & = \int_0^\infty e^{\alpha y} G(x,y) 2 \gamma y
    e^{-2\mu y^2}  dy \\  
    & = 2 \gamma \left[\wh\phi(x) \int_0^x y e^{ \alpha y}\wh\Psi(y) \, 
      dy + \wh\psi(x) \int_x^\infty y e^{\alpha y} \wh\Phi(y) \, dy \right].
  \end{align*}
  Using \eqref{eq:27} and \eqref{eq:43} we see that the term in the brackets is bounded by
 \begin{align*}
   K_1 K_2 & \Bigl[ e^{-\gamma x/\mu} \int_0^x y e^{\gamma y/\mu + \alpha y} \, dy + 
   e^{\gamma x/\mu + \mu x^2} \int_x^\infty y e^{-\gamma y/\mu-\mu y^2+\alpha
     y}\, dy \Bigr] 
 \end{align*}
 which itself can be bounded by $\wt C(\gamma,\mu) x e^{\alpha x} $ for some
 constant $\wt C(\gamma,\mu)$ independent of $x$.    
\end{proof}

\begin{lemma}[Expected killing position starting from $x$]
  \label{lem:kil-pos-OU} \leavevmode \\
  There is a positive finite constant $\wh c$ such that for all $x\geq 0$ 
  \begin{align}
    \label{eq:46}
    \mathbb E_x[\wh Z_{\wh\tau-}] \le x+\wh c.   
  \end{align}
\end{lemma}
\begin{proof}
We have 
\begin{align*}
  \mathbb E_x[\wh Z_{\wh\tau-}] & = 2 \gamma \int_0^\infty  y^2 G(x,y) e^{- \mu 
    y^2} \, dy\\  
  &= \wh{\phi}(x)\int_0^x 2 \gamma y^2 \wh{\Psi}(y) \, dy   
    +  \wh\psi(x) \int_x^\infty 2 \gamma y^2  \wh\Phi(y) \, dy. 
\end{align*}
Let us first consider the two integrals. The functions $\wh\Phi$ and $\wh\Psi$
satisfy \eqref{eq:tildeeq}, which can be rewritten as 
\begin{align}
  \label{eq:44}
  2\gamma x u(x) = u''(x) + 2\mu x u'(x) + 2 \mu u(x).  
\end{align}
Using this (twice in each computation) and partial integration we obtain  
\begin{align*} 
  \int_0^x 2 \gamma y^2 \wh{\Psi}(y) \, dy & = \int_0^x y \left(\wh{\Psi}''(y) +
    2\mu y \wh{\Psi}'(y) + 2\mu \wh{\Psi}(y) \right)\,  dy  \\
  & = x \wh{\Psi}'(x)- \wh{\Psi}(x)+ \wh{\Psi}(0)  + 2\mu x^2 \wh{\Psi}(x)
  -\frac{\mu}{\gamma} \wh{\Psi}'(x) - \frac{2\mu^2}{\gamma} x \wh{\Psi}(x) 
\end{align*}
and
\begin{align*}
  \int_x^\infty 2\gamma y^2  \wh{\Phi}(y) \, dy  & = \int_x^\infty y
  \left(\wh{\Phi}''(y) + 2 \mu y \wh{\Phi}'(y) + 2\mu \wh{\Phi}(y) \right) \, dy
  \\ & = -  x \wh{\Phi}'(x) + \wh{\Phi}(x) - 2\mu  x^2 \wh{\Phi}(x)+
  \frac{\mu}{\gamma} \wh{\Phi}'(x) + \frac{2\mu^2}{\gamma} x \wh{\Phi}(x). 
\end{align*}
Now from $\wh \psi \wh\Phi = \wh\phi \wh\Psi$, \eqref{eq:45} and \eqref{eq:47}
it follows  
\begin{align*}
  \mathbb E_x[Z_{\wh\tau-}]&= (x-\frac\mu\gamma) \left(\wh{\phi}(x)\wh\Psi'(x)
    -\wh\psi(x) \wh\Phi (x)\right)  + \wh\phi(x) \wh\Psi(0) + \frac\mu\gamma
    \wh\phi(x) \wh\Psi'(0) \\ & =  x-\frac\mu\gamma  + \wh\phi(x) \wh\Psi(0) %+
  %  \frac\mu\gamma  \wh\phi(x) \wh\Psi'(0)  
    \le x + \frac\mu\gamma  +  \wh\phi(0) \wh\Psi(0) % + \frac\mu\gamma \wh\phi(0) \wh\Psi'(0),    
  \end{align*}
  where the last inequality follows because $\wh\phi$ is 
  decreasing. This concludes the proof. 
\end{proof}

\begin{lemma}[Second moment of the killing time] 
  \label{lem:kill-time-OU} %\leavevmode \\    
  For all $x \ge 0$ we have $\E_x[\wh\tau^2] < \infty$.    
\end{lemma}
\begin{proof}
  As in the proof of Lemma~\ref{lem:2nd.mom.kill.time} it is enough to show that
  $\mathbb E_0[\wh\tau^2]$ is finite. We have   
  \begin{align} \label{eq:31} 
    \mathbb E_0[\wh\tau^2] & = 2 \int_0^\infty G(0,x) m(x) \int_0^\infty G(x,y) 
    m(y) \, dy \, dx. 
  \end{align}
  It follows that (recall the definition of $m$ in \eqref{eq:speed-killing-measure}) 
\begin{align*}
  \int_0^\infty G(x,y) m(y) \, dy & = \frac1\omega\wh\phi(x) \int_0^x \wh\psi(y) m(y)\,dy +
  \frac1\omega\wh\psi(x) \int_x^\infty \wh\phi(y) m(y) \,dy \\ 
  & \le \frac2\omega \wh\phi(x) K_1 \int_0^x e^{\gamma y/\mu}\, dy + \frac2\omega \wh\psi(x)K_2
  \int_x^\infty e^{-\gamma y/\mu - \mu y^2}\, dy \\ 
  & \le \frac2\omega \wh\phi(x) K_1 \int_0^x e^{\gamma y/\mu}\, dy + \frac2\omega \wh\psi(x)K_2
  e^{-\gamma x/\mu} \int_x^\infty e^{ - \mu y^2}\, dy.   
\end{align*}
Now using an estimate for $x \ge 0$ 
\begin{align*}
  \int_x^\infty e^{-\mu y^2} \, dy \le \frac1{\mu x+\sqrt{\mu}\sqrt{\mu
      x^2+4/\pi}} e^{-\mu x^2}  \le \frac{\sqrt{\pi}}{2 \sqrt \mu} e^{-\mu x^2} 
\end{align*}
(which can be deduced from \citep[][7.1.13]{AbramowitzStegun:1972}), 
and \eqref{eq:27} we obtain after simple calculations 
\begin{align*}
  \int_0^\infty G(x,y) m(y) \, dy   & \le \frac2\omega K_1K_2 ( e^{-\gamma
    x/\mu} \frac{\mu}{\gamma}(e^{\gamma x/\mu}-1) + \frac{\sqrt{\pi}}{2 \sqrt
    \mu} e^{\gamma x/\mu + \mu x^2} e^{-\gamma x/\mu - \mu x^2} ) \\  
    & \le C(\gamma,\mu), 
\end{align*}
for suitably chosen constant $C(\gamma,\mu)$ depending on $\gamma$ and $\mu$
but independent of $x$. Putting this into equation \eqref{eq:31} we get 
\begin{align*}
  \E_0[\wh\tau^2] \le 2 C^2(\gamma,\mu)   
\end{align*}
and the proof is completed. 
\end{proof}

\subsection{Invariant Distribution at jump times} 
\label{OUP:invariant distribution}

Also for the Ornstein-Uhlenbeck ratchet we consider the Markov chain of
increments at jump times.  
\begin{definition}[Markov chain at jump times] \label{eq:mkjt-ou} Let $(\wh\XX,
  \wh\RR)$ be a $(\gamma,\mu)$-Ornstein-Uhlenbeck ratchet with sequence of jump
  times of $\wh\RR$ given by $(\wh\tau_n)_{n\ge 0}$. At jump times we define the
  Markov chain $(\wh\YY,\wh\WW,\wh\eta) \coloneqq (\wh Y_n,\wh W_n,\wh
  \eta_n)_{n\ge 1} $ by
  \begin{align}
    \label{eq:mc.inc.OU}
    \wh Y_n= \wh X_{\wh \tau_n}-\wh R_{\wh \tau_n}, && \wh W_n = \wh
    R_{\wh\tau_n}-\wh R_{\wh \tau_{n-1}} && \text{and} &&  \wh \eta_n = \wh
    \tau_n-\wh\tau_{n-1}.  
  \end{align}
\end{definition}
Now, using the moment bounds in Lemma~\ref{lem:exp-mom-coupl-OUR},
Lemma~\ref{lem:kil-pos-OU} and Lemma~\ref{lem:kill-time-OU} the next result
follows as in the case of the Brownian ratchet. 
\begin{proposition}%[Existence and uniqueness of invariant distribution] 
  \label{prop:ex-uniq-inv.OUP}  
%  \leavevmode% 
%  \\
  There exists a unique invariant distribution of the Markov chain $(\mathcal Y,
  \mathcal W, \eta)$.   
\end{proposition}
Our next aim is to compute the moments under this invariant distribution. First
we show that the invariant density of the $\wh Y$ component satisfies a
differential equation.  
\begin{lemma}%[ODE for the density of invariant distribution] 
  \label{lem:our-density} %\leavevmode \\
  Let $\nu$ be the  invariant distribution of
  $\wh Y_1,\wh Y_2,\dots$ and let $\wh f_\nu$ be the corresponding density.
  Then $\wh f_\nu$ is the unique positive  decreasing solution of   
  \begin{align}
    \label{eq:ourdensity}
  \frac12 \wh f''_{\nu}(x) & = - \mu x\wh f'_\nu (x) + \gamma x \wh f_{\nu}(x) 
  \end{align}
  satisfying $\wh f_\nu''(0)=0$ and $\int_0^\infty \wh f_\nu(x)\, dx =1$. 
\end{lemma}
\begin{proof}
  Similar to the case of killed Brownian motion with negative drift one can
  write down a recurrence equation for $\wh f_\nu$. We have   
  \begin{align}
    \label{eq:ourdensity0}
    \begin{split} 
      \wh f_\nu(z)&=   \int_0^\infty \wh f_{\nu}(x) \int_z^\infty G(x,u) \,du \, dx 
    \end{split} \\ \intertext{and} 
    \label{eq:ourdensity1}
    \wh f'_\nu(z)&= 2\gamma \Bigl( - \wh \Psi(z) \int_z^\infty
    \wh f_\nu(x)  \wh\phi(x) \, dx - \wh \Phi (z) \int_0^z \wh f_\nu(x)  \wh\psi (x)  
    \,dx \Bigr)\\ 
    \label{eq:ourdensity2}
    \wh  f''_\nu(z)&=  2\gamma  \Bigl( - \wh \Psi'(z)
    \int_z^\infty \wh f_\nu(x)  \wh\phi(x) \, dx - \wh \Phi'(z) \int_0^z \wh
    f_\nu(x) \wh\psi(x) \, dx \Bigr)\\ 
    \label{eq:ourdensity3}
    \wh f'''_\nu(z)&= 2\gamma (x \wh f_{\nu}(x))' - 2\mu(x\wh f'_\nu (x))'.
  \end{align}
  The equations \eqref{eq:ourdensity0}--\eqref{eq:ourdensity2} are analogous to 
  the case of the Brownian ratchet with negative drift. We only give some
  details on how we obtain \eqref{eq:ourdensity3}.   

  Differentiating \eqref{eq:ourdensity2} and using 
  the fact that $\wh\Phi$ and $\wh \Psi$ solve \eqref{eq:tildeeq} we obtain
  \begin{align*} 
    \frac{1}{2\gamma} \wh f'''_{\nu}(x) & = - \wh \Psi''(x) \int_x^\infty
    \wh  f_{\nu}(z) \wh\phi(z) \, dz + \wh\Psi'(x) \wh f_\nu(x) \wh\phi(x) \\    
    & \quad - \wh \Phi''(x) \int_0^x \wh f_{\nu(z)}  \wh\psi(z) \, dz - \wh \Phi'(x)
    \wh f_\nu(x) \wh\psi(x)  \\   
    & = - 2 \bigl((\gamma x -\mu )\wh \Psi(x) - \mu x \wh\Psi'(x)\bigr)
    \int_x^\infty \wh f_{\nu}(z)  \wh\phi(z) \, dz +  \wh\Psi'(x) \wh f_{\nu}(x)
    \wh\phi(x) \\    
    & \quad -2 \bigl((\gamma x -\mu)\wh \Phi(x) - \mu x 
    \wh\Phi'(x)\bigr)  \int_0^x \wh f_{\nu}(z)  \wh\psi(z) \, dz - \wh\Phi'(x)
    \wh f_{\nu}(x)  \wh\psi(x)\\   
    & = \wh f_{\nu}(x)(\wh\Psi'(x)\wh\phi(x) -  \wh\Phi'(x)\wh\psi(x)) \\
    & \quad + 2 (\gamma x-\mu) \Bigl(- \wh \Psi(x) \int_x^\infty \wh f_{\nu}(z)
    \wh\phi(z) \, dz-\wh \Phi(x) \int_0^x \wh f_{\nu}(z) \wh\psi(z) \, dz \Bigr)\\ 
    & \quad +2\mu x \Bigl(\wh\Psi'(x) \int_x^\infty \wh f_{\nu}(z) \wh\phi(z)\,
    dz+ \wh\Phi'(x)  \int_0^x \wh f_{\nu}(z)  \wh\psi(z) \,  dz\Bigr).    
  \end{align*}
  Thus, using \eqref{eq:45}, \eqref{eq:ourdensity1} and \eqref{eq:ourdensity2}
  we obtain
  \begin{align*} 
    \wh f'''_{\nu}(x) & = 2 \gamma \wh f_{\nu}(x)+ 2(\gamma x- \mu) \wh
    f'_{\nu}(x) - 2\mu x \wh f''_{\nu}(x)\\ 
    & = 2 \gamma (x \wh f'_{\nu}(x) + \wh f_{\nu}(x)) -2 \mu(x\wh f''_{\nu}(x)+
    \wh f'_{\nu}(x))\\   
    & =  2 \gamma \bigl(x \wh f_{\nu}(x)\bigr)' -2 \mu\bigl(x \wh f'_\nu (x)\bigr)',  
  \end{align*}
  which shows \eqref{eq:ourdensity3}. 

  Now integrating $f'''_{\nu}(x)$ we obtain 
  \begin{align} \label{eq:dens-ode}
   \wh f''_{\nu}(x) & =  2 \gamma x \wh f_{\nu}(x) -2 \mu x \wh f'_\nu(x).    
  \end{align}
  Here, the integration constant is zero because by \eqref{eq:ourdensity2} and
  \eqref{eq:47} it follows that $\wh f''_{\nu}(0) = 0$. By \eqref{eq:ourdensity1}
  $x\mapsto \wh f_\nu(x)$ is decreasing. This concludes the proof.    
\end{proof}

Recall $p(x)$ in \eqref{eq:40}. One can check that the general solution of 
\eqref{eq:ourdensity} is given by  $e^{-\gamma x/\mu -\mu x^2}$ multiplied by a
linear combination of  
\begin{align} 
  \label{eq:49}
   M\left(-\frac{\gamma^2}{4\mu^3}+\frac12,\frac12,p^2(x)\right)\quad  \text{and}
   \quad 
   U\left(-\frac{\gamma^2}{4\mu^3}+\frac12,\frac12,p^2(x) \right)  \\   
  \intertext{if $\frac{\gamma^2}{4\mu^3} -\frac12 \ne 1, 2,\dots$, and}  
  \label{eq:50} 
   p(x) M \left(-\frac{\gamma^2}{4\mu^3}+1,\frac32,p^2(x) \right) \quad  \text{and}
   \quad  U \left(-\frac{\gamma^2}{4\mu^3}+\frac12,\frac12,p^2(x) \right) 
\end{align}
if $\frac{\gamma^2}{4\mu^3} -\frac12 = 1, 2,\dots$. In view of
\eqref{eq:M-asympt} the modulus of any solution containing a non-zero proportion
of $M$ in the above cases behaves (up to a polynomial factor) as $e^{\gamma
  x/\mu}$ as $x \to \infty$ and therefore cannot converge to $0$ for $x \to
\infty$. Thus, the density of the invariant distribution of the $\wh Y$ component of
the Markov chain $(\wh{\mathcal Y}, \wh{\mathcal W}, \wh \eta)$ is given by   
\begin{align}
  \label{eq:53}
  \wh f_\nu(x) & = \frac1{\int_0^\infty h_{\mu,\gamma}(y)\,dy} h_{\mu,\gamma}(x), \\ 
  \intertext{where} \label{eq:33} 
  h_{\mu,\gamma}(x) & \coloneqq    e^{-\gamma x/\mu -\mu x^2} U
  \left(-\frac{\gamma^2}{4\mu^3}+\frac12,\frac12,p^2(x) \right). 
\end{align}
The following result shows that $\wh f_\nu$ defined in \eqref{eq:53} satisfies
conditions stated in Lemma~\ref{lem:our-density}.  
\begin{lemma} 
We have 
\begin{enumerate}
\item $h_{\mu,\gamma}$ is positive on $[0,\infty)$,   
\item $h_{\mu,\gamma}''(0)=0$,  
\item $\int_0^\infty h_{\mu,\gamma}(x) \,dx<\infty$, 
\item $h'_{\mu,\gamma}(x) < 0$ for $x >0$.
\end{enumerate}
\end{lemma}
\begin{proof}
  Let $\gamma, \mu>0$ be given. Throughout the proof we write $h$ for
  $h_{\mu,\gamma}$. The assertions (ii) and (iii) hold because $h$ is a 
  solution of \eqref{eq:ourdensity} and because by \eqref{eq:U-asympt} $h(x)
  \sim e^{-\gamma x/\mu -\mu x^2}$ as $x \to \infty$.  

  To show (i) we first observe that by \eqref{eq:U-asympt} $h(x) >0$ for
  sufficiently large $x$. Thus, it is enough to show that $h$ has no positive
  zeros. By the Kummer transformation \eqref{eq:Kummer} % $U(a,b,x)=x^{1-b} U(a-b+1,2-b,x)$
  the function $h$ has positive zeros if and only if the function $x \mapsto
  U(1-\frac{\gamma^2}{4\mu^3},\frac32,p^2(x))$ has positive zeros which is not
  the case as we have seen in the proof of Lemma~\ref{lem:whphi-decr}. 

  To show (iv) we show that $h$ is convex, i.e.\ $h''(x) >0$ for $x >0$. Then 
  (iv) follows from (i) and from $h(x) \to 0$ as $x \to \infty$. For $x > 0$ we
  have  
  \begin{align*}
    \frac1{2x}  h''(x) & = -\mu h'(x) +\gamma h (x) \\
    & = -\mu \Bigl((-\frac\gamma\mu -2\mu x) h (x) \\ & \qquad + e^{-\gamma x
      /\mu -\mu x^2} 2 \sqrt\mu (\frac\gamma{\mu^{3/2}}+\sqrt\mu x) (-(\frac12
    -\frac{\gamma^2}{4\mu^3})) U (\frac32-\frac{\gamma^2}{4\mu^3},\frac32,p^2(x)
    )\Bigr) + \gamma h (x) \\ & = 2 (\gamma + \mu^2 x) \Bigl(h(x) +
    e^{-\gamma x /\mu -\mu x^2} (\frac12 -\frac{\gamma^2}{4\mu^3})U
    (\frac32-\frac{\gamma^2}{4\mu^3},\frac32,p^2(x) )\Bigr) \\ & = 2 (\gamma +
    \mu^2 x) e^{-\gamma x /\mu -\mu x^2} \Bigl(U
    (\frac12-\frac{\gamma^2}{4\mu^3},\frac12,p^2(x) ) + (\frac12
    -\frac{\gamma^2}{4\mu^3})U
    (\frac32-\frac{\gamma^2}{4\mu^3},\frac32,p^2(x) )\Bigr)\\
    & = 2 (\gamma + \mu^2 x) e^{-\gamma x /\mu -\mu x^2} U
    (\frac12-\frac{\gamma^2}{4\mu^3},\frac32,p^2(x)),
  \end{align*}
  where the last equality follows from \eqref{eq:51}. Now again by the Kummer
  transformation $h''$ is positive on $(0,\infty)$ if and only if the function
  $x \mapsto U (-\frac{\gamma^2}{4\mu^3},\frac12,p^2(x))$ is positive on that
  interval. This was shown in Lemma~\ref{lem:whphi-decr}. 
\end{proof}

\subsection{Proof of Theorem~\ref{theoremOU}}
\label{sec:proof34}

The proof in the case of the Ornstein-Uhlenbeck ratchet is almost the same as in
the case of the Brownian ratchet with negative drift. That is, we can again
define a sequence of regeneration times and show that the temporal and spatial
increments of the ratchet between this regeneration times have finite second
moments. All the ingredients needed for the proof of that have been provided in
the previous subsections. We content ourselves with computation of the speed
of the ratchet. To this end we need (as in Proposition~\ref{prop:exp-under-pi})
to compute the expectation of $\wh Y_1$ and of $\wh \eta_1$ under the invariant
distribution $\nu$.

Using \eqref{eq:ourdensity} we obtain 
\begin{align*}
  \E_\nu[\wh Y_1] & = \int_0^\infty x \wh f_\nu (x) \, dx = \frac1{2\gamma}
  \int_0^\infty \left(\wh f''_\nu(x) + 2\mu x \wh f_\nu'(x)\right) \, dx  \\ 
  & = \frac1{2\gamma} \left(- \wh f'_\nu(0) - 2\mu \int_0^\infty \wh f_\nu(x)\, dx\right) =
  \frac1{2\gamma} \left(- \wh f'_\nu(0) - 2\mu\right).  
\end{align*}
Furthermore, recalling \eqref{eq:38}, we have using Fubini's Theorem in the second
equality and \eqref{eq:ourdensity1} in the third  
\begin{align*}
  \E_\nu[\wh\eta_1] & = 2 \int_0^\infty \wh f_\nu(x) \left(\wh \phi(x) \int_0^x
    \wh\Psi(y)\,dy +
    \wh\psi(x) \int_x^\infty \wh\Phi(y)\, dy \right) \, dx \\
  & = 2 \int_0^\infty \left(\wh\Psi(y) \int_y^\infty \wh f_\nu(x) \wh \phi(x)
    \,dx + \wh\Phi(y) \int_0^y \wh f_\nu(x) \wh\psi(x) \, dx \right) \, dy \\
  &= - \frac1\gamma  \int_0^\infty \wh f_\nu'(y) \, dy  = \frac{\wh
    f_\nu(0)}{\gamma}. 
\end{align*}
Now the speed of the Ornstein-Uhlenbeck ratchet is given by 
\begin{align*}
  \wh v (\mu,\gamma) \coloneqq \frac{\E_\nu[\wh Y_1]}{\E_\nu[\wh\eta_1]}=
  - \frac{\wh f'_\nu(0) + 2\mu}{2\wh f_\nu(0)} =  - \frac{h'_{\mu,\gamma}(0)}{2h_{\mu,\gamma}(0)} -
  \frac{\mu \int_0^\infty h_{\mu,\gamma}(x)\, dx }{h_{\mu,\gamma}(0)}.
\end{align*}
\qed

 \section*{Acknowledgements}
 The authors would like to thank Peter Pfaffelhuber and Martin Kolb for fruitful
 discussions.

%\nocite{0483.33002,MR1052433}
%\bibliographystyle{alea3-mod}
%\bibliography{references}

\end{document}